\documentclass{article}
\usepackage[utf8]{inputenc}

\usepackage{amssymb}
\usepackage{amsmath,mathtools}
\usepackage{amsthm}
\usepackage{amsfonts}
\usepackage{tikz}
\usepackage{xcolor}
\usepackage[top=1.0in, bottom=1.0in, left=1.0in, right=1.0in, headheight=1.0in]{geometry}
\usepackage{mathtools}
\usepackage{youngtab}
\usepackage{graphicx}
\usepackage{multicol}
\usepackage{paracol}
\graphicspath{ {./Graphs/} }
\usepackage[nottoc,notlof,notlot]{tocbibind}
\usepackage{changepage}

\newtheorem{theorem}{Theorem}[section]
\newtheorem{lemma}[theorem]{Lemma}
\newtheorem{definition}[theorem]{Definition}

\newtheorem{corollary}[theorem]{Corollary}
\newtheorem{remark}[theorem]{Remark}

\newcommand{\I}{\mathcal{I}}
\newcommand{\U}{\mathcal{U}}
\newcommand{\NN}{\mathcal{N}}

\newcommand{\inv}{I}
\newcommand{\GI}{G^\inv}

\newcommand{\KSG}{K \cup S}

\newcommand{\cross}{\text{cdeg}}

\newcommand{\crosspair}{\{\alpha, \beta\}}
\newcommand{\redpair}{\{\alpha', \beta'\}}

\date{\today  }
\title{The Hereditary Closure of the Unigraphs}

\author{Michael D. Barrus\\
\small Dept. of Mathematics and Applied Mathematical Sciences\\
\small University of Rhode Island\\
\small Kingston, RI 02281\\
\small\tt barrus@uri.edu\
\and
Ann N. Trenk\thanks{
This work was supported by a grant from the Simons Foundation (\#426725, Ann Trenk). } \\
\small Department of Mathematics\\
\small Wellesley College\\
\small Wellesley, MA 02481\\
\small\tt atrenk@wellesley.edu\
\and
Rebecca Whitman\\
\small Department of Mathematics\\
\small University of California, Berkeley\\
\small Berkeley, CA 94720\\
\small\tt  rebecca\_whitman@berkeley.edu\
}
\begin{document}
\maketitle

\begin{abstract}    
A graph with degree sequence $\pi$ is a \emph{unigraph} if it is isomorphic to every graph that has degree sequence $\pi$. The class of unigraphs is not hereditary and in this paper we study the related hereditary class HCU, the hereditary closure of unigraphs, consisting of all graphs induced in a unigraph. We characterize the class HCU in multiple ways making use of the tools of a decomposition due to Tyshkevich and a partial order on degree sequences due to Rao. We also provide a new characterization of the class that consists of unigraphs for which all induced subgraphs are also unigraphs.
  \end{abstract}
  
  \smallskip
  

\section{Introduction}

In this paper we consider two classes (which we will call $HU$ and $HCU$) of finite, simple graphs that are characterized in terms of possible realizations of their degree sequences. Recall that the \emph{degree sequence} of a graph is the list of degrees of the graph's vertices, written in non-increasing order. A graph $G$ with a degree sequence $\pi$ is said to \emph{realize} $\pi$ or to be a \emph{realization} of $\pi$. When a realization exists for a finite list $\rho$ of non-negative integers, then we call $\rho$ \emph{graphic}.

A \emph{unigraph} is a graph that is the unique graph up to isomorphism having its specific degree sequence. For example, each graph on four or fewer vertices is a unigraph. On the other hand, neither $P_5$ nor $K_2+K_3$ is a unigraph, since these non-isomorphic graphs share the degree sequence $(2,2,2,1,1)$. (Here we use $+$ to denote a disjoint union, and $P_n$ or $K_n$ to denote a path or complete graph, respectively, on $n$ vertices.) Degree sequences will be called \emph{unigraphic} or \emph{non-unigraphic} depending on whether they have only one realization or multiple, respectively.

Unigraphs are rare (see~\cite{Ty80,Ty00} for counts due to Tyshkevich) and special. The ability to precisely describe each unigraph by listing its vertex degrees simplifies some combinatorial problems, and unigraphs can be used to examine the effect of degree sequence-based bounds on graph parameters (see, for example, \cite{Ba12b}, where unigraphs highlight the tightness of the degree sequence residue as a bound on the independence numbers of graph realizations). Early work on unigraphs and their degree sequences was done by many authors (see \cite{Jo75,Jo80,KlLi75,Ko76a,Ko76b,Li75}, for several important works), culminating in a complete characterization by Tyshkevich and Chernyak (see \cite{TyCh78,TyCh79a, TyCh79b, TyCh79c} and \cite{Ty00}).

The class of unigraphs contains several families of graphs that are remarkable in their own right. These include the families of threshold graphs~\cite{HaIbSi81}, matroidal graphs, and matrogenic graphs~\cite{MaEtAl84,Ty84}. See the references for these families' definitions. With one exception, each of these families of unigraphs, including the family of unigraphs itself, has three distinct types of characterizations. To begin with, each has a characterization in terms of degree sequences (see~\cite{HaIbSi81,KlLi75,Ko76a,Ko76b,Li75,MaEtAl84,Ty84}). This may not be surprising, given that all graphs involved are unigraphs. 

As a second type of characterization, each of the classes of threshold graphs, matroidal graphs, matrogenic graphs, and general unigraphs has a characterization in terms of graph structure. The members of each family can be recognized as having subgraphs of certain prescribed types, connected by rigidly prescribed patterns of adjacency; see~\cite{ChHa75,FoHa78,MaPe95,Pe77,Ty00}.

Finally, the threshold graphs, matroidal graphs, and matrogenic graphs each can be characterized in terms of a list of forbidden induced subgraphs; see~\cite{ChHa75,FoHa78,Pe77}. This is possible because each of these families of graphs is a \emph{hereditary} family, i.e., one that is closed under taking induced subgraphs of its members.

The full class of unigraphs, however, is not hereditary. As one example, observe that though $(3,3,3,3,3,1)$ has a unique realization $G$, the graph $G$ contains two non-unigraphs as induced subgraphs (namely, the two realizations of $(3,2,2,2,1)$). 

Given this difference between the family of unigraphs and those of its remarkable subfamilies, in~\cite{Ba12} and~\cite{Ba13} Barrus defined the family of \emph{hereditary unigraphs} to be the largest subfamily of the unigraphs that is closed under taking induced subgraphs. In this paper we denote this family by $HU$. By definition, $HU$ does not contain the unigraph with degree sequence $(3,3,3,3,3,1)$ mentioned above, but it does contain each of the threshold, matroidal, and matrogenic graphs. Barrus showed that like these families, $HU$ has characterizations in terms of the degree sequences and structure of its elements, and $HU$ has a characterization in terms of a finite number of forbidden induced subgraphs. 

We now ask about a complementary question. The family $HU$ is the largest hereditary family contained by the (non-hereditary) family of unigraphs; what does the \emph{smallest} hereditary family \emph{containing} the family of unigraphs look like?

We define the \emph{hereditary closure of the unigraphs}, which we denote henceforth by $HCU$, as the collection of all graphs that occur as an induced subgraph of some unigraph. Under this definition the two realizations of $(3,2,2,2,1)$ belong to $HCU$, as do all unigraphs. In this paper we show that $HCU$, like its subfamilies described above, has characterizations in terms of degree sequences, structural conditions, and a finite list of forbidden induced subgraphs.

Along the way, we discuss a new type of characterization for graphs in $HU$ and $HCU$. This characterization relies on a partial order on degree sequences introduced by S.B. Rao in~\cite{Rao80}. The result is a hybrid incorporating elements of degree sequences and forbidden induced subgraphs.

The paper is structured as follows. In Section 2 we discuss tools that aid in our characterizations; these include ideas about split graphs and the scheme of graph composition introduced by Tyshkevich in~\cite{Ty00}. In Section 3 we provide a structural characterization for graphs in $HCU$, extending the collection of ``building blocks'' introduced in~\cite{Ty00} to allow for non-unigraph members in $HCU$. In Section 4 we provide the first of our characterizations of $HCU$ that involve the degree sequences of its elements. In preparation for another degree sequence characterization, in Section 5 we introduce Rao's degree sequence poset and develop some preliminary results; we also give a new characterization of $HU$ involving the Rao poset. In Section 6 we establish a degree sequence characterization of $HCU$ in terms of the Rao poset. Finally, in Section 7 we describe a computational search resulting in the list of minimal forbidden induced subgraphs for the family $HCU$.

\section{Preliminaries and Tyshkevich decomposition}

In this section we collect preliminary notions and describe the Tyshkevich decomposition of graphs, which will serve as a structural framework in all of our characterizations of graphs in $HCU$. In our approaches to characterizing $HCU$ we will often deal separately with its split and non-split elements. Split graphs were first defined by Gyarfas and Lehel in \cite{GyLe69}, and independently by F\"{o}ldes and Hammer in \cite{FoHa77}. We recall some known preliminaries here; for a more comprehensive treatment, see \cite{CoTr21}.

\begin{definition}
\label{def:split}
    {\rm A} split graph {\rm is a graph in which the vertex set can be partitioned into a clique and a stable set (also known as an independent set). For a split graph $G$, the partition $V(G) = K \cup S$ is called a} $KS$-partition {\rm if $K$ is a clique and $S$ is a stable set.}
\end{definition}

The class of split graphs is hereditary, since any induced subgraph inherits a $KS$-partition from the original graph. Membership in the class of split graphs can be characterized by degree sequence \cite{HaSi81}. Accordingly, we call a degree sequence \emph{split} if its realizations are split graphs, and \emph{non-split} otherwise. 

In a split graph $G$, we call a vertex $v$ a \emph{swing vertex} if there exist two  $KS$-partitions of $G$ such that $v$ is in the clique of one and the stable set of the other (with the idea that $v$ can ``swing" back and forth between the clique and the stable set). Equivalently, $v$ is a swing vertex if its degree is equal to the clique number of $G$. A split graph has a unique $KS$-partition if and only if it has no swing vertices. 

Given a split graph $G$ with a fixed $KS$-partition, the \emph{inverse} of $G$, denoted $G^I$, is the graph obtained by adding to the graph all possible edges with both endpoints in $S$ and removing from the graph all edges with both endpoints in $K$. In this way $G^I$ is also a split graph, with $S$ and $K$ partitioning its vertex set into a clique and a stable set, respectively. Note that our definition depends on a specified  $KS$-partition for $G$; if $G$ has more than one such partition, more than one inverse will exist. Context should make the partition clear, though we will often deal with split graphs with unique $KS$-partitions, and in this instance the inverse is unique.

We shall denote the complement of a graph $G$ by $\overline{G}$. Note that if $G$ is a split graph with $KS$-partition $K \cup S$ then the complement $\overline{G}$ is also a split graph, with $S$ and $K$ serving as the clique and stable set, respectively, in a $KS$-partition of $\overline{G}$. If $G$ is a split graph with fixed $KS$-partition, then the complement of the inverse of $G$ and the inverse of the complement of $G$ are equal, and we denote this graph by $\overline{\GI}$. Tracing the effects of complementation and taking inverses of distinct realizations of a degree sequence shows the following result. 

\begin{remark}
\label{prop:unigraph_complement_inverse}
    \rm If $U$ is a unigraph, then $\overline{U}$ is also a unigraph. If $U$ is also a split graph and $K \cup S$ is a $KS$-partition of $U$, then both $U^\inv$ and $\overline{U^\inv}$ are unigraphs. 
\end{remark}

Tyshkevich \cite{Ty00} introduced a way to express any graph as a composition of split graphs and at most one non-split graph. We first define the composition of a split graph with another graph. The composition of a split graph $G$ with fixed $KS$-partition and another graph $H$ is achieved by taking the disjoint union of $G$ and $H$ and joining each vertex of $H$ to each vertex of $K$. We make this precise in the next definition. 

\begin{definition}
\label{def:composition}
    {\rm Let $G$ be a split graph with $KS$-partition $\KSG$, and let $H$ be any graph. The} (Tyshkevich) composition of $G$ and $H$ {\rm is the graph on $V(G) \cup V(H)$ with edge set $E(G) \cup E(H) \cup \{uv|u \in K, v\in V(H)\}$. We denote this by $(G, K, S) \circ H$, and write $G \circ H$ when the $KS$-partition of $G$ is clear.}
\end{definition}

\begin{figure}
\centering
  \includegraphics[height=3.5cm]{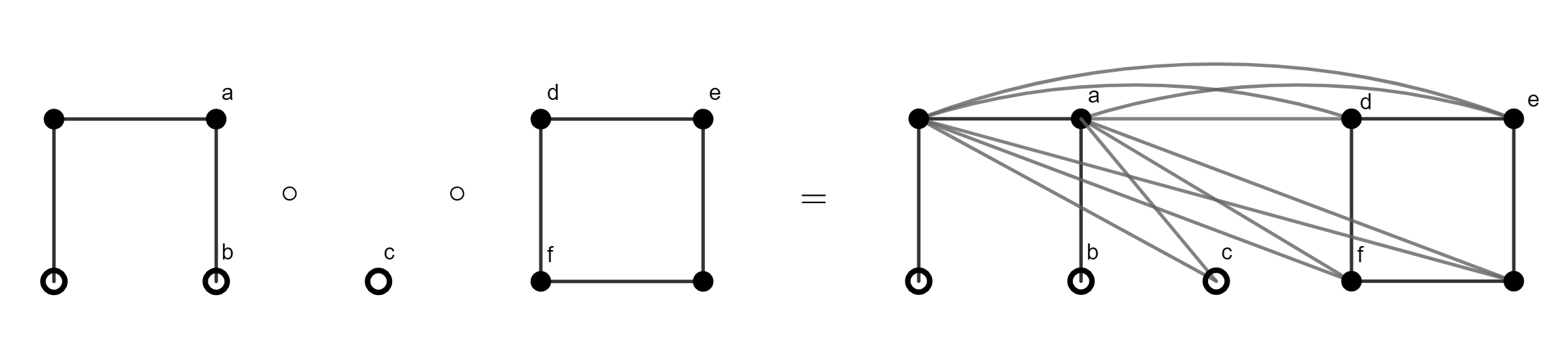}
  \caption{The composition $P_4 \circ \overline{K_1} \circ C_4$. Vertices of $P_4$ and $\overline{K_1}$ drawn on the top row are in the clique portion of the partitions and those on the bottom row are in the stable set. }
  \label{fig:composition}
\end{figure}

It is not hard to show that the composition $G \circ H$ is a split graph if and only if $H$ is a split graph. Using this, it is straightforward to extend the composition to $n$ split graphs and one graph chosen freely, and to observe that this operation is associative. Figure \ref{fig:composition} shows the composition of $P_4$, which has a unique $KS$-partition; the single vertex labeled $c$, which is placed in the $S$ part of its $KS$-partition; and the non-split graph $C_4$. (We use $C_n$ to denote a cycle graph on $n$ vertices, and other standard graph naming notation throughout). Throughout the paper, we draw clique vertices of split graphs with a solid circle, and stable set vertices with an open circle. The abstract structure of Tyshkevich composition is shown later in Figure \ref{fig:composition_degree}. 

Graphs that cannot be expressed as a non-trivial graph composition play a fundamental role in our characterization theorems. 

\begin{definition}
\label{def:indecomposable}
    {\rm A graph $G$ is} decomposable {\rm if there exist nonempty graphs $G_1, G_0$, where $G_1$ is a split graph and $G=G_1 \circ G_0$, and} indecomposable {\rm otherwise.}
\end{definition}

Notice that a graph $G$ on at least two vertices with an isolated vertex $v$ is decomposable, with vertex $v$ as a leftmost component that has $K=\emptyset$ and $S=\{v\}$. With the opposite $KS$-partition, graphs with a dominating vertex are similarly decomposable, as are split graphs with a swing vertex, where the rightmost component is a singleton vertex. A consequence of this last observation is that indecomposable split graphs have no swing vertices and thus have exactly one $KS$-partition.

\begin{remark}
\label{prop:indecomposable_DIS}
    \rm Indecomposable graphs with at least two vertices contain no dominating, isolated, or swing vertices. 
\end{remark}

We state the existence and uniqueness of Tyshkevich decomposition; a proof is found in \cite{Ty00}. 

\begin{theorem} \cite{Ty00}
\label{prop:Introduction:decomposition}
    Every graph $G$ can be written as a composition of indecomposable components, where each of $G_n, \ldots, G_1$ has a fixed $KS$-partition: 
    $$G=G_n \circ \ldots \circ G_1 \circ G_0.$$
    Furthermore, when each $G_i$ is nonempty, this decomposition is unique up to isomorphism, where neither the order of the components nor the choice of $KS$-partitions can vary. 
\end{theorem}

When we refer to the \textit{decomposition} $G=G_n \circ \ldots \circ G_0$, we mean the Tyshkevich decomposition of $G$ into its \textit{indecomposable} parts as guaranteed by Theorem \ref{prop:Introduction:decomposition}, whereas a (Tyshkevich) composition $H=H_n \circ \ldots \circ H_0$ is defined even when not all $H_i$ are indecomposable. 

We note that both complementation and inversion commute with composition, as well as with one another. Whereas the complement of a composition is taken component-wise, the inverse of a composition flips the order of the components. Therefore, the property of being indecomposable is shared by a graph and its complement, and if the graph is a split graph, it is also shared by its inverse and inverse-complement. We record these as remarks. 

\begin{remark}
\label{prop:commute_complement_inverse}
    \rm If $G= G_n \circ \ldots \circ G_0$, then $\overline{G} = \overline{G_n} \circ \ldots \circ \overline{G_0}$, with cliques and stable sets reversed. If $G_0$ is a split graph with fixed $KS$-partition, then the inverse $G^I$ with respect to the $KS$-partition implied by $G_n \circ \ldots \circ G_0$ is given by $G^I = G_0^I \circ \ldots \circ G_n^I$, reversing the order of the components. 
\end{remark}

\begin{remark}
\label{prop:indecomposable_complement_inverse}
    \rm The complement of an indecomposable graph is indecomposable, and the inverse of an indecomposable split graph is indecomposable. 
\end{remark}

Although the class of unigraphs has no forbidden induced subgraph characterization, it can be characterized via decomposition.

\begin{theorem} \cite{Ty00}
\label{prop:Introduction:Tyshkevich_unigraph}
    A graph $G$ is a unigraph if and only if each component of its decomposition is a unigraph. 
\end{theorem}

To understand the set of unigraphs, using Theorem \ref{prop:Introduction:Tyshkevich_unigraph}, it suffices to determine which indecomposable graphs are unigraphs. Let $\U$ be the set of indecomposable unigraphs. Tyshkevich characterized $\U$ as a set of two graphs and seven families; we state this characterization in Theorem \ref{prop:indecomposable_unigraphs}, using streamlined notation for each graph family, then give a brief description of each family. The individual graphs and graph families are shown in Figures \ref{fig:NonSplit} and \ref{fig:Split}. We also collect the names of several graph classes in Table~\ref{table:classes}; the definitions of some of these classes are forthcoming in this and the next section.

For $m \ge 2$, the graph $T_1(m)$ is the disjoint union of $m$ copies of $K_2$. For $m \ge 1$ and $n \ge 2$, the graph $T_2(m,n)$ is the disjoint union of $T_1(m)$ and the $n$-star $K_{n,1}$, for $m \ge 1$ and $n \ge 2$. The graph $T_3(m)$ is formed by attaching one $C_4$ and $m \ge 1$ copies of $C_3$ at a shared vertex. 

\begin{figure}[p]
\centering
  \includegraphics[height=5cm]{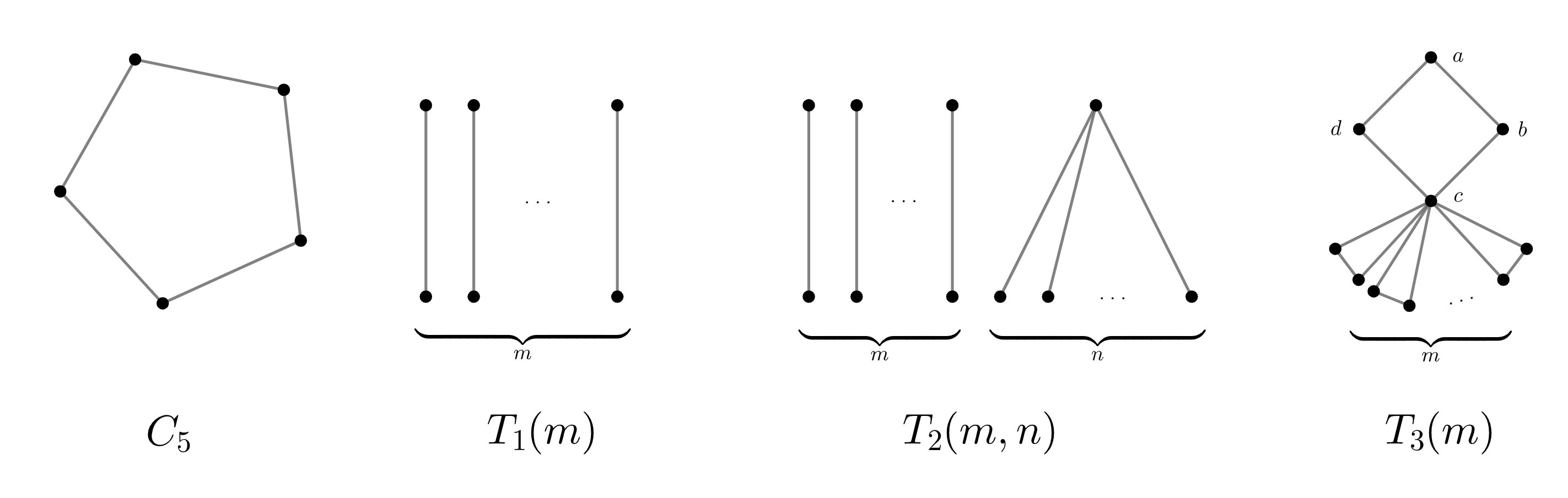}
  \caption{The non-split graphs in $\U$.}
  \label{fig:NonSplit}
\end{figure}

The graph $T_4(p,q)$ consists of $q \ge 2$ clique vertices, each of which is adjacent to exactly $p \ge 1$ leaves. For $p_i \ge 1, q_i \ge 1$, and $r \ge 2$, the graph $T_5(p_1, q_1, \ldots, p_r, q_r)$ has $q_1 + q_2 + \ldots + q_r$ clique vertices, and each of $q_i$ clique vertices are adjacent to exactly $p_i$ leaves. For $p \ge 1, q_1 \ge 1,$ and $q_2 \ge 1$, the graph $T_6(p, q_1, q_2)$ is a split graph with a clique containing $q_1+q_2$ vertices. Each clique vertex has $p+1$ neighbors in the stable set; $q_2$ of them are adjacent to $p+1$ leaves, and $q_1$ of them are adjacent to $p$ leaves and a shared stable set vertex, $e$. For $p \ge 1$ and $q \ge 1$, the graph $T_7(p,q)$ is formed from $T_6(p, 2, q)$ by adding a vertex $f$ to the clique, adjacent to all stable set vertices except $e$. 

\begin{figure}[p]
\centering
  \includegraphics[height=9cm]{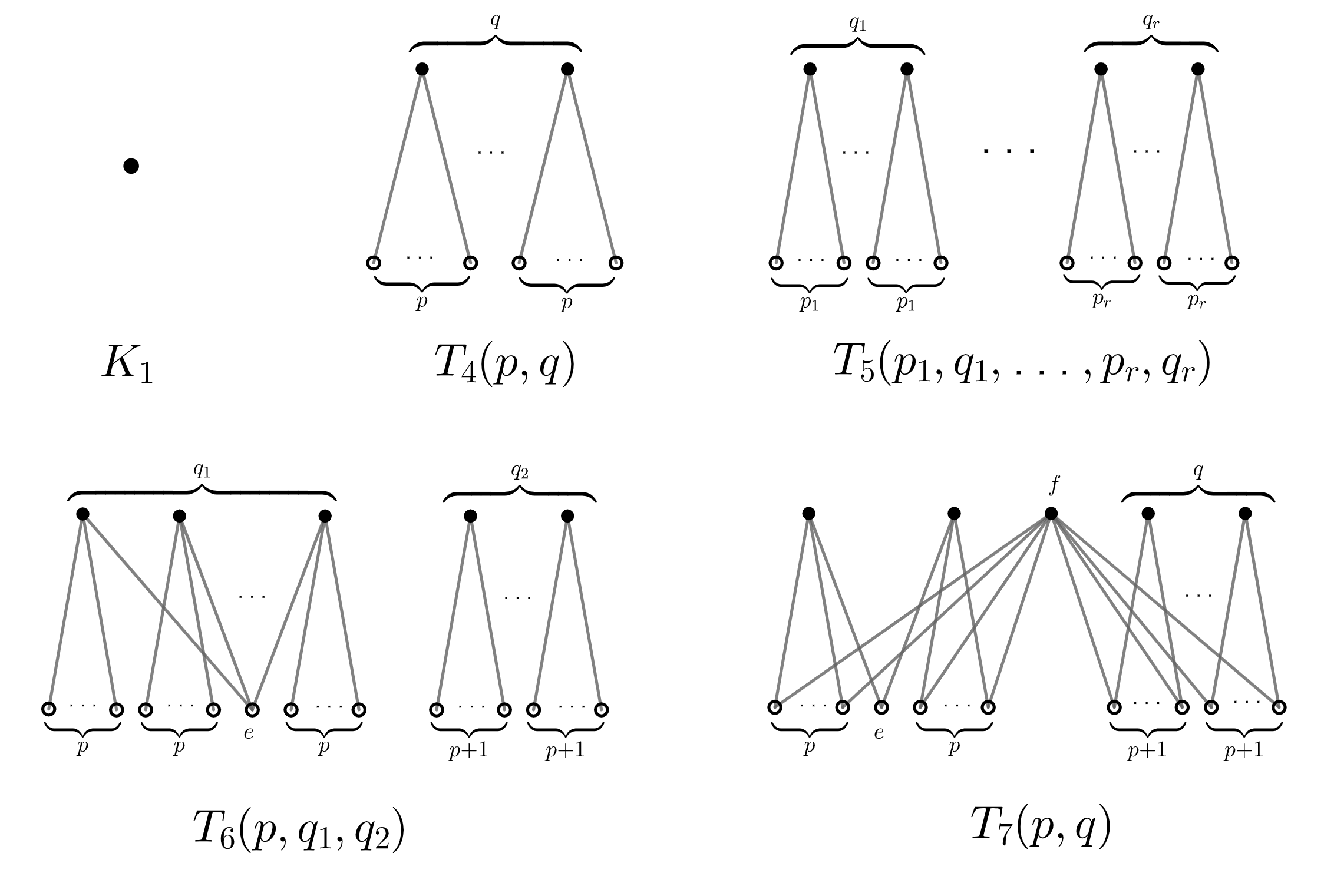}
  \caption{The split graphs in $\U$. The solid vertices in each graph's upper row form the graph's clique, although clique edges are not shown. The lower row of hollow vertices in each graph are the stable set vertices.}
  \label{fig:Split}
\end{figure}

For $1 \le i \le 7$, the family ${{\mathcal{T}}_i}$ contains exactly the graphs $T_i$ with acceptable choices of parameters. 

\begin{theorem} \cite{Ty00}
\label{prop:indecomposable_unigraphs}
The indecomposable non-split unigraphs are $C_5$ and the families ${\mathcal{T}_1}$, ${\mathcal{T}_2}$, and ${\mathcal{T}_3}$, as well as the complements of each graph in these families. The indecomposable split unigraphs are $K_1$ and the families ${\mathcal{T}_4}$, ${\mathcal{T}_5}$, ${\mathcal{T}_6}$, and ${\mathcal{T}_7}$, as well as the complements, inverses, and inverse-complements of each graph in these families. 
\end{theorem}

\begin{table}
    \begin{tabular}{| l | l |}
        \hline
         $HU$ & Unigraphs for which all induced subgraphs are unigraphs \\
         $HCU$ & Smallest hereditary class containing all unigraphs; equivalently, graphs induced in a unigraph\\
         $\U$ & Indecomposable unigraphs \\ 
         $\I$ & Indecomposable graphs induced in a unigraph \\  
         $\NN$ & Non-unigraphs in $\I$  \\  
        \hline
    \end{tabular}
    \caption{Several subclasses of $HCU$. }
    \label{table:classes}
\end{table}

\section{A structural characterization of $HCU$}

In this section, we show that to characterize the class $HCU$, it suffices to characterize the indecomposable graphs induced in a unigraph. We begin with a lemma describing the interaction of induced subgraphs and compositions. 

If $G$ is written as a composition $G = G_n \circ \ldots \circ G_0$ and $H$ is an induced subgraph of $G$, then removing the vertices of $G$ not in $H$ and maintaining the same $KS$-partitions in each split graph component provides a representation of $H$ as a composition. This is recorded in Lemma \ref{prop:Characterization:induced_composition} below and illustrated in Figure \ref{fig:composition}, where the graph induced in $G = P_4 \circ \overline{K_1} \circ C_4$ by the labeled vertices $a$-$f$ is equal to the composition $K_2 \circ \overline{K_1} \circ P_3$, with $KS$-partitions inherited from each component on the left-hand side. 

\begin{lemma}
\label{prop:Characterization:induced_composition}
    Let $G_0$ be a graph, let $G_1, \ldots, G_n$ be split graphs and let $K_i \cup S_i$ be a $KS$-partition of $G_i$ for $1 \le i \le n$. Further, let $H$ be an induced subgraph of $G$. Define $H_i$ to be the graph induced in $H$ by the vertices of $G_i$, and if $i \ge 1$, assign $H_i$ the $KS$-partition inherited from $G_i$. Then $H = H_n \circ \ldots \circ H_0$. 
\end{lemma}

Lemma \ref{prop:Characterization:induced_composition} holds regardless of whether any components $G_i$ or $H_i$ are decomposable. 

Let $\I$ be the set of indecomposable graphs induced in a unigraph, that is, $H \in \I$ if and only if $H$ is indecomposable and there exists a unigraph $G$ for which $H$ is induced in $G$. We first characterize $HCU$ in terms of $\I$, then return to a characterization of $\I$. 

\begin{theorem}
\label{prop:Characterization:characterization}
    A graph is in $HCU$ if and only if each graph in its decomposition is in $\I$. 
\end{theorem}
\begin{proof}
    First, let $G$ be a graph with decomposition $G_n \circ \ldots \circ G_0$ where $G_i \in \I$ for $0 \le i \le n$. By definition of $\I$, there exists a unigraph $U_i$ containing $G_i$ as an induced subgraph. By the definition of composition, $G=G_n \circ \ldots \circ G_0$ is induced in $U_n \circ \ldots \circ U_0$, which is a unigraph by Theorem \ref{prop:Introduction:Tyshkevich_unigraph}. Hence $G$ is induced in a unigraph and therefore $G \in HCU$. 

    For the reverse direction, let $G \in HCU$; thus there exists a unigraph $U$ with $G$ induced in $U$. Let $U = U_n \circ \ldots \circ U_0$ be the decomposition of $U$, so each $U_i$ is an indecomposable unigraph by Theorem \ref{prop:Introduction:Tyshkevich_unigraph}. Let $G_i = G \cap U_i$, thus each $G_i$ is induced in $U_i$ and by Lemma \ref{prop:Characterization:induced_composition}, $G=G_n \circ \ldots \circ G_0$. Since the graph $G_i$ is induced in $U_i$, we know for all $i$ that $G_i$ is an induced subgraph of an indecomposable unigraph. However, $G_i$ might not be indecomposable. For each $i$, let $G_{i,m_i} \circ \ldots \circ G_{i,0}$ be the decomposition of $G_i$. Now, for all $j \le m_i$, $G_{i,j}$ is indecomposable and induced in $G_i$, which in turn is induced in $U_i$. Therefore $G_{i,j}$ is induced in $U_i$, and so $G_{i,j} \in \I$ for all $(i,j)$. By substitution, the decomposition of $G$ into indecomposable parts is $G=(G_{n,m_n} \circ \ldots \circ G_{n,0}) \circ \text{\LARGE\ensuremath \ldots} \circ (G_{0,m_0} \circ \ldots \circ G_{0,0}).$ Since the composition relation $\circ$ is associative, we conclude that the decomposition of $G$ is $G=G_{n,m_n} \circ \ldots \circ G_{0,0}$, and each of these components is in $\I$. 
\end{proof}

Recall from Remark \ref{prop:commute_complement_inverse} that complements and inverses both commute with composition, albeit with reordered components. In conjunction with this observation, we conclude from Theorem \ref{prop:Characterization:characterization} that $HCU$ is closed under complementation, and the subset of split graphs within $HCU$ is closed under inversion. 

\begin{corollary} 
\label{prop:HCU_closed_complements}
    A graph $G$ is in $HCU$ if and only if $\overline{G}$ is in $HCU$. If $G$ is a split graph and $K \cup S$ is a $KS$-partition of $G$, then either all or none of $G, \overline{G}, \GI, \overline{\GI}$ are in $HCU$. 
\end{corollary}

This implies the same of minimal forbidden induced subgraphs of $HCU$. Furthermore, such graphs must be indecomposable, else some proper indecomposable component of a forbidden graph must also be forbidden. 

\begin{corollary}
\label{prop:MFIS_indecomp_closed}
    If $G$ is a minimal forbidden induced subgraph of $HCU$, then $G $ is indecomposable, and $\overline{G}, \GI$, and $\overline{\GI}$ are each also minimal forbidden induced subgraphs. 
\end{corollary}

We now present a characterization of $\I$. First we show we need only consider graphs induced in \textit{indecomposable} unigraphs.

\begin{lemma}
\label{prop:induced_indecomposable}
    If an indecomposable graph is induced in a unigraph, then it is also induced in an indecomposable unigraph.
\end{lemma}
\begin{proof}
Let $G$ be an indecomposable graph induced in a unigraph $U$. By Theorem \ref{prop:Introduction:Tyshkevich_unigraph}, each component $U_i$ of the decomposition $U = U_n \circ \ldots \circ U_0$ is also a unigraph. Now by Lemma \ref{prop:Characterization:induced_composition}, we can write $G$ as a composition $G = G_n \circ \ldots \circ G_0$ where $G_i$ is the graph induced in $G$ by $V(U_i)$. Since $G$ is indecomposable, Definition \ref{def:indecomposable} implies that there is only one nonempty component $G_j$ in this composition, so $G= G_j$. Thus $G$ is induced in the indecomposable unigraph $U_j$. 
\end{proof}

Let $\NN$ be the set of non-unigraphs in $\I$. Thus $\I = \U \cup \NN$, and by Lemma \ref{prop:induced_indecomposable}, $\NN$ is the class of indecomposable non-unigraphs induced in a graph in $\U$. The class $\NN$ consists of five graph families and was characterized in \cite{Wh20}. We present this result in Theorem \ref{prop:indecomposable_nonunigraphs} and give a sketch of the proof. 

The proof of Theorem \ref{prop:indecomposable_nonunigraphs} uses the next lemma to identify indecomposable graphs. Lemma \ref{prop:indecomposable_A4} follows from the observation that in a composition, vertices inducing the indecomposable graphs $P_4, C_4,$ and $2K_2$ cannot be separated into multiple indecomposable components. A more general result can be found in Barrus and West \cite{BaWe12}.

\begin{lemma} \cite{BaWe12}
\label{prop:indecomposable_A4}
    If $G$ is a graph and there exists a vertex $v \in V(G)$ such that for all $w \in V(G)$, there exists an induced $P_4, C_4, $ or $2K_2$ containing both $v$ and $w$, then $G$ is indecomposable. 
\end{lemma}

\begin{theorem}
\label{prop:indecomposable_nonunigraphs}
    A graph $G \in \NN$ if and only if $G$ contains no dominating, isolated, or swing vertices, and at least one of the following holds for $G$, its complement, or, if $G$ is a split graph, its inverse or inverse-complement: (See Figures \ref{fig:NonSplit} and \ref{fig:Split} and definitions preceding Theorem \ref{prop:indecomposable_unigraphs}.)
    \begin{enumerate}
        \item [\rm{(1)}] $G$ is induced in a graph in $\mathcal{T}_3$; $a,b,c,d \in V(G)$; and $c$ has a neighbor of degree $1$.
        \item [\rm{(2)}] $G$ is induced in a graph in $\mathcal{T}_3$; $a,c \in V(G)$; exactly one of $b$ and $d$ is in $V(G)$; and $G$ contains a $K_3$ subgraph.
        \item [\rm{(3)}] $G$ is induced in a graph in ${\mathcal{T}_6}$; $e \in V(G)$ with $\deg{e} \ge 2$; and there exists a pair of clique vertices of unequal degree, exactly one of which is adjacent to $e$. 
        \item [\rm{(4)}] $G$ is induced in a graph in $\mathcal{T}_7$; $e \in V(G)$ with $\deg{e}=2$; and there exists a pair of clique vertices of unequal degree, exactly one of which is adjacent to $e$. 
        \item [\rm{(5)}] $G$ is induced in a graph in $\mathcal{T}_7$; $e, f, x_1 \in V(G)$; there exists a vertex of degree $1$ in $G$; and some vertex of degree $2$ has a neighbor other than $x_1$ and $f$. 
    \end{enumerate}
\end{theorem}
\begin{proof}
If $G \in \NN$, then by Lemma \ref{prop:induced_indecomposable}, $G$ is an indecomposable non-unigraph and $G$ is induced in a graph listed in Theorem \ref{prop:indecomposable_unigraphs}. It is not hard to show that the graphs $C_5$, $K_1$, and those in ${\mathcal{T}_1}, {\mathcal{T}_2}, {\mathcal{T}_4},$ and ${\mathcal{T}_5}$ are in $HU$, so we need only consider when $G$, its complement, and, if $G$ is a split graph, its inverse or inverse-complement, are induced in graphs in ${\mathcal{T}_3}$, ${\mathcal{T}_6}$, or ${\mathcal{T}_7}$. Without loss of generality, suppose $G$ itself is induced in one of these families. 

We provide a detailed argument in the case that $G$ is induced in $U = T_3(m)$. We refer to the vertex labels that are in Figure \ref{fig:NonSplit}. If $c \not\in V(G)$, then $G$ is induced in a graph in ${\mathcal{T}_2}$ and so is a unigraph. Thus $c \in V(G)$. If $a \not\in V(G)$, then $c$ is a dominating vertex and $G$ is decomposable, contradicting $G \in \NN$. Thus $a \in V(G)$. Similarly, vertex $a$ cannot be an isolated vertex, so without loss of generality we assume $b \in V(G)$. From here, there are four possibilities. First, if $G$ contains neither $C_4$ nor any $C_3$, then $G \in {\mathcal{T}_5}$ and $G$ is an indecomposable unigraph by Theorem \ref{prop:indecomposable_unigraphs}. Next, if $G$ contains some $C_3$, but not $C_4$, then $G$ satisfies (2). If $G$ contains $C_4$, but $c$ has no degree $1$ neighbors, then either $G=C_4$ or $G=T_3(k)$ for some $k$, and is a unigraph in either case. Lastly, if $G$ contains $C_4$ and $c$ has at least one degree $1$ neighbor, then $G$ satisfies (1). 

We omit the cases where $G$ is induced in a graph in ${\mathcal{T}_6}$ or ${\mathcal{T}_7}$. 

For the reverse direction, let $G$ be a graph without dominating, isolated, or swing vertices, and so that one of $G$, its complement, and, if $G$ is a split graph, its inverse and inverse-complement satisfy one of (1) to (5). Suppose without loss of generality that $G$ itself satisfies one of these. We provide a detailed argument in the case that $G$ satisfies (1). In this case, $G$ is induced in $U \in {\mathcal{T}_3}$, the subset $\{a,b,c,d\}$ induces $C_4$, and $c$ has at least one degree $1$ neighbor $v$. It suffices to show $G$ is a non-unigraph and indecomposable. Create a graph $G'$ from $G$ by deleting edge $ad$ and adding edge $vd$. Note that $G$ and $G'$ have the same degree sequence but are not isomorphic, since $G'$ has no induced $C_4$. We conclude that $G$ is a non-unigraph. 

Next, we show that $G$ is indecomposable. The set of vertices $\{a,b,c,d\}$ induces $C_4$ in $G$ and any other vertex $w \in V(G)$ together with vertices $a,b,c$ induces $P_4$. Thus every vertex of $G$ is part of an induced $P_4$ or $C_4$ containing vertex $c$, so by Lemma \ref{prop:indecomposable_A4}, $G$ is indecomposable.

We omit proofs of the remaining cases, which are similar. 
\end{proof}

As stated above, $\I = \U \cup \NN$, so our characterization of $\NN$ in Theorem \ref{prop:indecomposable_nonunigraphs} and Tyshkevich's characterization of $\U$ in Theorem \ref{prop:indecomposable_unigraphs} together give a complete list of the graphs in $\I$. By Theorem \ref{prop:Characterization:characterization}, the composition of elements in $\I$ gives us exactly the class $HCU$.

\section{A degree sequence characterization of $HCU$}

In this section, we will characterize the degree sequences of graphs in $HCU$. We begin by showing that membership in the class $HCU$ is determined by degree sequence, ensuring the existence of such a characterization. 

\begin{lemma}
\label{prop:HCU_closed_degree}
    If graphs $G$ and $G'$ have the same degree sequence and $G \in HCU$, then $G' \in HCU$. 
\end{lemma}
\begin{proof}
    Suppose $G$ and $G'$ are graphs with the same degree sequence and $G \in HCU$. Then there exists a unigraph $U$ with $G$ induced in $U$. Replace the edges in $U$ that have both endpoints in $G$ with the edges in $G'$ so that vertex degrees are unchanged. The result is a graph $U'$ that contains $G'$ as an induced subgraph, and $U'$ has the same degree sequence as $U$. Since $U$ is a unigraph, it is isomorphic to $U'$ and thus $G' \in HCU.$
\end{proof}

The next result from Tyshkevich \cite{Ty00} shows that when two graphs have the same degree sequence, so do the corresponding components of their decomposition. 

\begin{theorem} \cite{Ty00}
\label{prop:degree_sequence_decomposition}
    Suppose $G$ is a graph with decomposition $G_n \circ \ldots \circ G_0$ and $H$ is a graph with decomposition $H_m \circ \ldots \circ H_0$. If $G$ and $H$ have the same degree sequence, then $n=m$ and the degree sequence of $G_i$ equals the degree sequence of $H_i$ for all $i,$ $0 \le i \le n$. 
\end{theorem}

Following Tyshkevich, Theorem \ref{prop:degree_sequence_decomposition} allows us to define the decomposition of a graphic sequence irrespective of the choice of realization, and to define decomposable and indecomposable degree sequences. 

\begin{definition}
\label{def:degree_seq_decomposition}
    {\rm Let $\pi$ be a graphic sequence, let $G$ be a graph realizing $\pi$, and let $G_n \circ \ldots \circ G_0$ be the decomposition of $G$. Define the} (Tyshkevich) decomposition of $\pi$ {\rm to be $\pi^n \circ \ldots \circ \pi^0$ where $\pi^i$ is the degree sequence of $G^i$ for each $i$. If $n = 0$ then $\pi$ is} indecomposable {\rm and if $n \ge 1$ then $\pi$ is} decomposable.
\end{definition}

Furthermore, we can discern from the decomposition of a degree sequence whether its realizations are in $HCU$.

\begin{lemma} 
\label{prop:HCU_closed_degree_components}
    Let $\pi$ be a degree sequence with decomposition $\pi^n \circ \ldots \circ \pi^0$. Each realization of $\pi$ is in $HCU$ if and only if each of the components $\pi^i$ is realized by a graph in $HCU$. 
\end{lemma}
\begin{proof}
    Suppose $\pi$ is a graphic sequence with realization $G$. Let $G_n \circ \ldots \circ G_0$ and $\pi^n \circ \ldots \circ \pi^0$ be the decompositions of $G$ and $\pi$, respectively. By Definition \ref{def:degree_seq_decomposition}, $\pi^i$ is the degree sequence of $G_i$ for each $i$. For the forward direction, assume that $G \in HCU$. Since $HCU$ is a hereditary class and $G_i$ is induced in $G$, it follows that $G_i \in HCU$ for each $i$. 
    Conversely, suppose each $\pi^i$ is realized by a graph $H_i \in HCU$. Since $G_i$ and $H_i$ have the same degree sequence for each $i$ and $H_i \in HCU$, we know $G_i \in HCU$ by Lemma \ref{prop:HCU_closed_degree}. Each $G_i$ is indecomposable, so $G_i \in \I$ by definition of $\I$. Now Theorem \ref{prop:Characterization:characterization} implies $G \in HCU$ as desired. 
\end{proof}

The indecomposable components of a degree sequence can be efficiently determined from the degree sequence itself (see~\cite{Ty00} and also~\cite{Ba13}). For convenience we record here a helpful fact about Tyshkevich decomposition from~\cite{Ba13}.

\begin{lemma}[See~{\cite[Observation 5.2]{Ba13}} and {\cite[Theorem 5.6]{Ba13}}]
\label{prop:EG_nonsplit_useful}
    Let $\pi = (d_1, \ldots, d_n)$ be a graphic sequence with $t=\max\{i|d_i \ge i-1\}$. Then $\pi$ is decomposable if and only if $d_n=0$, $d_n = n-1$, or there exists $k<t$ with $d_k > d_{k+1}$ such that the $k$th Erd\H{o}s-Gallai inequality holds with equality. 
\end{lemma}

We characterize when indecomposable sequences are realized by graphs in $HCU$ separately for split and non-split sequences, beginning with the latter. Theorems \ref{prop:degree_characterization_nonsplit} and \ref{prop:degree_characterization_split} both show that the degree sequences of graphs in $HCU$ are extreme cases: for instance, degree sequences satisfying (1) of the following theorem have all but one term equal to $1$, the smallest possible value in an indecomposable sequence. 

\begin{theorem} 
\label{prop:degree_characterization_nonsplit}
  All realizations of a non-split, graphic sequence $\pi=(d_1, \ldots, d_n)$ are indecomposable and in $HCU$ if and only if $d_n \ge 1, d_1 \le n-2$, and at least one of the following holds: 
    \begin{multicols}{2}
    \begin{enumerate}
        \item [\rm{(1)}] $d_2 = 1$
        \item [\rm{(2)}] $d_{n-1}=n-2$
        \item [\rm{(3)}] $d_1 = n-2$ and $d_2 = 2$
        \item [\rm{(4)}] $d_n = 1$ and $d_{n-1} = n-3$
        \item [\rm{(5)}] $d = (2,2,2,2,2)$
    \end{enumerate} 
    \end{multicols}
\end{theorem}

\begin{proof}
For the forward direction, let $G$ be an indecomposable non-split graph in $HCU$ with degree sequence $\pi = (d_1, \ldots, d_n)$. Since $G$ is indecomposable, it has no dominating or isolated vertices, so $d_n \ge 1$ and $d_1 \le n-2$. By Lemma \ref{prop:induced_indecomposable}, the graph $G$ is induced in an indecomposable unigraph $U$, and since $G$ is a non-split graph, so is $U$. The indecomposable non-split unigraphs are characterized in Theorem \ref{prop:indecomposable_nonunigraphs} and depicted in Figure \ref{fig:NonSplit}. We consider each possibility for $U$. If $U = C_5$, then the only possibility for $G$ is $C_5$ itself, and $\pi$ satisfies (5). It is not hard to verify that if $G$ is induced in a graph in ${\mathcal{T}}_1$ or ${\mathcal{T}}_2$ then $\pi$ satisfies (1), and if $G$ is induced in their complements, then $\pi$ satisfies (2). We can similarly see that if $G$ is induced in a graph in ${\mathcal{T}}_3$ or its complement, for some $m \ge 1$, then $\pi$ satisfies (3) or (4), respectively.

For the backwards direction, suppose $\pi = (d_1, \ldots, d_n)$ is a non-split, graphic sequence with $d_n \ge 1$, $d_1 \le n-2$, and satisfying at least one of conditions (1) to (5). 
Let $G$ be a graph realizing $\pi$ with $V(G)=\{v_1, \ldots, v_n\}$ and $\deg{v_i}=d_i$ for all $i$,  $i \le n$. We will show that $G$ is an indecomposable graph in $HCU$. 

First suppose $\pi$ satisfies (1), so $\pi = (d_1, 1, 1, \ldots, 1)$. If $d_1 = 1$, then $n$ is even and $G$ is isomorphic to $\frac{n}{2}K_2$. If $d_1 \ge 2$, then $G$ is isomorphic to $T_2(\frac{n-1-d_1}{2}, d_1)$, where $n-1-d_1$ is an even integer (because $\pi$ is graphic) and at least two (because $G$ is not a split graph). In either case, $G$ is an indecomposable unigraph by Theorem \ref{prop:indecomposable_unigraphs}. 

Next, suppose $\pi$ satisfies (2), so $\pi = (n-2, \ldots, n-2, d_n)$. In this case, the degree sequence of $\overline{G}$ satisfies (1), so $\overline{G}$ is an indecomposable unigraph by the above. By Remarks \ref{prop:unigraph_complement_inverse} and \ref{prop:indecomposable_complement_inverse}, $G$ is an indecomposable unigraph. 

Now, suppose $\pi$ satisfies (3), so $\pi = (n-2, 2, \ldots, 2, 1, \ldots, 1)$. Let $a$ be the number of terms in $\pi$ equal to $1$. We will show that $G$ is induced in the unigraph $U=T_3(\frac{n+a-4}{2})$, and it may be helpful to refer to Figure \ref{fig:NonSplit}. Since $d_1 = n-2$, it follows that $v_1$ (labeled $c$ in Figure \ref{fig:NonSplit}) is adjacent to all other vertices except one, $v_s$ (labeled $a$). If $\deg{v_s}=2$, then $v_s$, its neighbors, and $v_1$ induce $C_4$. Otherwise, $\deg{v_s}=1$, and then $v_s$, $v_1$, and their mutual neighbor induce $P_3$. In either case, the remaining degree two vertices must each be adjacent to $v_1$, so in pairs they form induced $C_3$ graphs with $v_1$ in $G$. The vertices of degree $1$ (other than $v_s$) must be adjacent to $v_1$ and for each degree $1$ vertex $v_j$, add a new vertex $w_j$ of degree $2$ adjacent to $v_1$ and $v_j$. If $\deg(v_s) = 1$, then $w_s$ completes an induced $C_4$ and in either case, the remaining $w_j$ form induced $C_3$ graphs with $v_1$ and $v_j$. The result is the unigraph $U$ containing $G$ as an induced subgraph, so $G \in HCU$. 

It remains to show $G$ is indecomposable when $\pi$ satisfies (2). If $d_3 = 1$ then $t=\max\{i|d_i \ge i-1\} = 2$ and if $d_3 = 2$ then $t = 3$. In either case, the only index less than $t$ for which $d_k > d_{k+1}$ is $k = 1$. However, $n-2 = d_1 < 1(0) + \sum_{i=2}^n \min\{1, d_i\} = n-1$, so the $1$st Erd\H{o}s-Gallai inequality holds strictly. By assumption, $1 \le d_n \le n-2$, and we conclude from Lemma \ref{prop:EG_nonsplit_useful} and Remark \ref{prop:indecomposable_DIS} that $d$ is indecomposable, implying $G$ is as well. 

When $\pi$ satisfies (4) the degree sequence of $\overline{G}$ satisfies (3), so as before, $G$ is an indecomposable graph in $HCU$. Finally, when $\pi$ satisfies (5), then $G = C_5$, which is an indecomposable unigraph. 
\end{proof}

We now turn to the characterization of degree sequences of split graphs in $HCU$. For split graphs with a fixed $KS$-partition, it will be helpful to consider only the edges with one endpoint in $K$ and the other in $S$, and for each vertex, the number of such edges incident to it. 

\begin{definition}
\label{def:cross_degree}
    {\rm If $G$ is a split graph with a fixed $KS$-partition, the} cross-degree \rm{of vertex $v$, denoted by $\cross(v)$, is defined as follows:
    
        If $v \in K$ then $\cross(v) = \deg(v) - (|K|-1).$
        
        If $v \in S$ then $\cross(v) = \deg(v)$.}
\end{definition}

For example, the split graph $P_4$ has a unique $KS$-partition and each vertex has cross-degree $1$. Generally, the vertices of any split graph with no swing vertices will have unique cross-degrees. For split graphs with a swing vertex, the cross-degrees depend on the $KS$-partition. We introduce the cross-degree sequence pair of a split graph, which we define with respect to a given $KS$-partition of the graph. 

\begin{definition}
\label{def:cross_degree_sequence}
    {\rm Suppose $G$ is a split graph and fix a $KS$-partition of $G$. Let $K = \{v_1, \ldots, v_n\}$ with $\cross(v_i) \ge \cross(v_{i+1})$ for all $i \le n-1$, and let $S = \{w_1, \ldots, w_m\}$ with $\cross(w_i) \ge \cross(w_{i+1})$ for all $i \le m-1$. Let $a_i = \cross(v_i)$ and $b_j = \cross(w_j)$ for all $i \le n$ and $j \le m$. The} cross-degree sequence pair {\rm of $G$ with the given $KS$-partition is given by the unordered pair $\crosspair$ with $\alpha = (a_1, \ldots, a_n)$ and $\beta = (b_1, \ldots, b_m)\}$.}
\end{definition}

For example, the cross-degree sequence pair of $P_4$ is $\{(1,1),(1,1)\}$. Using Definition \ref{def:cross_degree}, we can easily convert between cross-degree sequence pairs and split graphic (degree) sequences, although this correspondence is not generally bijective. The most useful property of the cross-degree of a vertex and the cross-degree sequence pair of a split graph is that these are invariant under inversion, with respect to a given $KS$-partition. As a result, graphs with a unique $KS$-partition, including all indecomposable split graphs, have a unique cross-degree sequence pair. By conversion to degree sequences, it is straightforward to show that split graphs in $HCU$ can be characterized by their cross-degree sequence pairs, and to give well-defined notions of decomposition and indecomposability. 

Lemma \ref{prop:EG_nonsplit_useful} is also easily translated to cross-degree sequence pairs by using a second inequality, proven independently by Gale \cite{Ga57} and Ryser \cite{Ry57}, in lieu of the Erd\H{o}s-Gallai inequality. We leave the proof to the reader. 

\begin{lemma}
\label{prop:cross_degree_GaleRyser}
    A cross-degree sequence pair $\crosspair=\{(a_1, \ldots, a_n),(b_1, \ldots, b_m)\}$ is decomposable if and only if $a_n = 0, a_n = m$, or there exists $k$ such that $a_k > a_{k+1}$ and $\crosspair$ satisfies with equality the $k$th Gale-Ryser inequality $$\sum_{i=1}^k a_i \le \sum_{i=1}^m \min\{k,b_i\}.$$
    Moreover, since $\crosspair$ is unordered, $\crosspair$ is decomposable if and only if $b_m = 0$, $b_m = n$, or there exists $k$ such that $b_k > b_{k+1}$ and $\sum_{i=1}^k b_i \le \sum_{i=1}^n \min\{k,a_i\}$ holds with equality.
\end{lemma}

We characterize all indecomposable split graphs in the hereditary closure by the cross-degrees of the vertices in a given partition. The use of the cross-degree, rather than ordinary degree sequences, draws out the almost identical relationship between the cases of Theorems \ref{prop:degree_characterization_nonsplit} and \ref{prop:degree_characterization_split}. Here, (1a) and (1b) mirror (1) of Theorem \ref{prop:degree_characterization_nonsplit}, although the non-split graph $C_5$ has no parallel. In addition, (3a) and (3b) describe cross-degree sequence pairs where one term is as large as possible, and the terms on the other side are all $1$ or $2$, just as (3) in Theorem \ref{prop:degree_characterization_nonsplit} characterizes degree sequences where one term is as large as possible and all remaining terms are $1$ or $2$. In both theorems, the complements behave similarly: (2a) and (2b) together mirror (2), and (4a) and (4b) mirror (4). 

\begin{theorem} 
\label{prop:degree_characterization_split}
All realizations of a cross-degree sequence pair $\crosspair=\{(a_1, \ldots, a_n),(b_1, \ldots, b_m)\}$ are indecomposable and in $HCU$ if and only if $a_n, b_m \ge 1$, $a_1 \le m-1$, $b_1 \le n-1$, and at least one of the following holds:
    \begin{multicols}{2}
    \begin{enumerate}
        \item [\rm{(1a)}] $a_2=1$. 
        \item [\rm{(2a)}] $a_{n-1}=m-1$. 
        \item [\rm{(3a)}] $b_1=n-1$ and $a_1 \le 2$.
        \item [\rm{(4a)}] $b_m=1$ and $a_n \ge m-2$.
        \item [\rm{(1b)}] $b_2=1$. 
        \item [\rm{(2b)}] $b_{m-1}=n-1$. 
        \item [\rm{(3b)}] $a_1=m-1$ and $b_1 \le 2$.
        \item [\rm{(4b)}] $a_n=1$ and $b_m \ge n-2$.
    \end{enumerate}
    \end{multicols}
\end{theorem}

\begin{proof}
For the forward direction, let $G$ be an indecomposable split graph in $HCU$. By Remark \ref{prop:indecomposable_DIS} we know $G$ has no dominating, isolated, or swing vertices, hence $G$ has a unique $KS$-partition and unique cross-degree sequence pair $\crosspair$. The same remark tells us that $\alpha$ and $\beta$ each have at least two terms, and $a_1 \le m-1$, $b_1 \le n-1$, and $a_n, b_m \ge 1$. Since the conditions we seek to prove are symmetric with respect to $\alpha$ and $\beta$, we may assume without loss of generality that $K = \{v_1, \ldots, v_n\}$ and $S = \{w_1, \ldots, w_m\}$ where $\cross(v_i) = a_i$ and $\cross(w_j) = b_j$. By Lemma \ref{prop:induced_indecomposable}, $G$ is induced in an indecomposable unigraph $U \in \U$. We know from Theorem \ref{prop:indecomposable_nonunigraphs} that graphs in $\NN$ that are induced in non-split graphs in $\U$ are themselves non-split graphs. We can therefore choose $U$ to be a split graph. The split graphs in $\U$ are characterized in Theorem \ref{prop:indecomposable_unigraphs} and depicted in Figure \ref{fig:Split}. We consider each possibility for $U$. 

First, suppose $U$ is in ${{\mathcal{T}}}_4$ or ${{\mathcal{T}}_5}$. In this case, every vertex in the stable set $S$ has degree $1$, so $b_2 = 1$ and $\crosspair$ satisfies (1b). If $U$ is the inverse of a graph in ${\mathcal{T}_4}$ or ${\mathcal{T}_5}$, then $\crosspair$ satisfies (1a). If $U$ is the complement of a graph in ${\mathcal{T}_4}$ or ${\mathcal{T}_5}$ then $\crosspair$ satisfies (2b), and finally if $G$ is the inverse-complement of a graph in ${\mathcal{T}_4}$ or ${\mathcal{T}_5}$, then $\crosspair$ satisfies (2a). Next suppose $U$ is in ${\mathcal{T}_6}$. Then every vertex in $S$ has degree $1$, except possibly $w_1$, so $b_2=1$ and $\crosspair$ satisfies (1b). As above, if $U$ is the inverse, complement, or inverse-complement of a graph in ${\mathcal{T}_6}$, then $\crosspair$ satisfies (1a), (2b), or (2a), respectively. 

Finally, suppose $U$ is in ${\mathcal{T}_7}$ and refer to Figure $\ref{fig:Split}$. If $f \not \in V(G)$, then $G$ is induced in a graph in ${\mathcal{T}_6}$, and the previous case applies. If $f \in V(G)$ and $e \not \in V(G)$, then $f$ is a dominating vertex in $G$, a contradiction. It remains to consider when $e, f \in V(G)$. All vertices in $S$ have degree $1$ or $2$ and $f \in K$ is adjacent to all vertices in $S$ except $e$. Thus $a_1 = m-1$ and $b_1 \le 2$, and $\crosspair$ satisfies (3b). As before, if $U$ is the inverse, complement, or inverse-complement of a graph in ${\mathcal{T}_7}$, then $\crosspair$ satisfies (3a), (4b), or (4a), respectively.  

For the reverse direction, suppose $\crosspair = \{(a_1, \ldots, a_n), (b_1, \ldots, b_m)\}$ is a cross-degree sequence pair with $a_n, b_m \ge 1, a_1 \le m-1$, $b_1 \le n-1$, and satisfying at least one of the eight conditions. Let $G$ be a graph realizing $\crosspair$, with split partition $V(G)=V \cup W$, $V=\{v_1, \ldots, v_n\}$ and $W=\{w_1, \ldots, w_m\}$ such that $\cross(v_i)=a_i$ and $\cross(w_j)=b_j$ for $i \le n$ and $j \le m$. One of $V,W$ is a clique and the other is a stable set. We will show that $G$ is in $HCU$ and is indecomposable.

First suppose that $\crosspair$ satisfies (1a), thus vertices $v_2, \ldots, v_n$ each have cross-degree $1$ in $G$. If $a_1 = 1$ then $G$ is a graph in the family ${\mathcal{T}_5}$, hence is an indecomposable unigraph in $HCU$. Otherwise we will show that $G$ is induced in a unigraph $U \in {\mathcal{T}_6}$, where $W$ is the clique in the split partition of $W(G)$. It may be helpful to refer to Figure \ref{fig:Split}. For each $w_j \in W$, add $b_1 - b_j$ vertices to $V$, each of cross-degree $1$ and adjacent to $w_j$. This increases the cross-degree of every vertex in $W$ to $b_1$. Each vertex of $V$ other than $v_1$ (labeled $e$ in Figure \ref{fig:Split}) has cross-degree $1$ and the result is $U = T_6(b_1-1, a_1, m-a_1)$. The only descent in $\alpha$ occurs at index $1$, but by assumption $a_1 \le m-1$ so the $1$st Gale-Ryser inequality is strict: $a_1 < \sum_{i=1}^m \min\{1, b_i\} = m$. Since $1 \le a_n \le m-1$ by assumption, it follows by Lemma \ref{prop:cross_degree_GaleRyser} that $\crosspair$ is indecomposable and hence $G$ is as well. If $\crosspair$ satisfies (2a), then $\overline{G}$ satisfies (1a), so $G$ is indecomposable and contained in the indecomposable unigraph $\overline{U}$. 

Next suppose $\crosspair$ satisfies (3a). The case $a_2=1$ was considered in (1a), so we may assume $a_2 \ge 2$ and thus $\alpha = (2, 2, \ldots, 2, 1, \ldots, 1)$. Let $a$ be the larger of $b_2$ and the number of $1$s in $\alpha$. Since $b_1 = n-1$, there exists a unique vertex $v_j \in V$ such that $w_1$ is adjacent to all vertices in $V$ except $v_j$. We will show $G$ is induced in the unigraph $U = T_7(a - 1, m-2)$, where $W$ is the clique in the split partition of $V(G)$. Again, it is helpful to refer to Figure \ref{fig:Split}. First, add an additional vertex $w_{m+1}$ to $W$ with $w_{m+1}$ adjacent to every vertex in $V$ that has cross-degree equal to $1$. As a result, every vertex in $V$ now has cross-degree equal to $2$. Let $W' = W \cup \{w_{m+1}\}$. For each $w_i \in W'$ with $2 \le i \le m+1$, add $a - b_i$ vertices to $V$, where each added vertex has cross-degree $2$ and is adjacent to $w_i$ and $w_1$. As a result, each vertex in $W'$ other than $w_1$ has cross-degree equal to $a$. Let $V'$ be the set $V$ together with these added vertices. Since the cross-degree of $w_1$ is $|V'| -1$, the resulting split graph $U$ (where $V'$ is a stable set and $W'$ is a clique) contains $G$ as an induced subgraph and is equal to $T_7(a - 1, m-2)$. Here, $v_j$ is the vertex labeled $e$ and $w_1$ is the vertex labeled $f$ in Figure \ref{fig:Split}. Thus $G \in HCU$. 

We show that $G$ is indecomposable. Since all terms in $\alpha$ are less than or equal to two, the $k$th Gale-Ryser inequality is strict for $2 \le k \le m-1$, as we have $\sum_{i=1}^n \min\{k,a_i\} = \sum_{i=1}^n a_i = \sum_{i=1}^m b_i > \sum_{i=1}^k b_i$. The assumption $b_1 \le n-1$ implies that the $1$st Gale-Ryser inequality also holds strictly, so we conclude from Lemma \ref{prop:cross_degree_GaleRyser} that $\crosspair$ is indecomposable, thus $G$ is indecomposable. If $\crosspair$ satisfies (4a) then $\overline{G}$ satisfies (3a), so $G$ is contained in the indecomposable unigraph $\overline{U}$. 

Finally, if $\crosspair$ satisfies (1b), (2b), (3b), or (4b), the above arguments apply with the roles of $V$ and $W$ reversed and the graphs and unigraphs each replaced by their inverses. 
\end{proof}

\section{The Rao poset on degree sequences}
\label{section:Rao_poset}

In Section~\ref{sec: Rao-Minimal} we will provide another characterization of the class $HCU$ in terms of degree sequences. This characterization will be strongly linked to a forbidden induced subgraph characterization, which will follow in Section~\ref{sec: forbidden subgraphs}. The context that reveals this link is a partially ordered set on degree sequences introduced by S.B.~Rao that incorporates information about degree sequences of induced subgraphs.

In \cite{Rao80}, Rao defined a partial order on the set of degree sequences of simple graphs with the following relation: degree sequences $\pi,\rho$ satisfy $\pi \preceq \rho$ if there exists a realization $G$ of $\pi$ and a realization $H$ of $\rho$ such that $G$ is an induced subgraph of $H$.  We refer to $\preceq$ as the \emph{Rao relation}, and we may also write $\rho \succeq \pi$. We say that $\pi$ \emph{is Rao-contained in} $\rho$ or that $\rho$ \emph{Rao-contains} $\pi$ if  $\pi \preceq \rho$. Note that if $\rho$ does not Rao-contain $\pi$, then no realization of $\rho$ contains any realization of $\pi$ as an induced subgraph.

As a simple example, observe that $(2,2,2) \preceq (2,2,2,1,1)$, since $(2,2,2)$ is realized by $K_3$, and $(2,2,2,1,1)$ has $K_3+K_2$ as one of its realizations. (The fact that the other realization, a path on five vertices, does not contain $K_3$ as an induced subgraph is irrelevant.) Hence the Rao poset captures, in the context of degree sequences, some of the relationships that the induced subgraph order does for graphs. 

In~\cite{Rao80} Rao showed that $\preceq$ is indeed a partial order and described some of its properties and related questions. Of note, Rao asked whether the poset so defined is a well-quasi-order (WQO), i.e., whether for any infinite sequence $\pi^1, \pi^2, \dots$ of degree sequences, there would exist indices $i,j \in \mathbb{N}$ such that $i < j$ and $\pi^i \preceq \pi^j$. Equivalently, the partial order would permit no infinite set of pairwise incomparable degree sequences. This question was answered many years later in~\cite{ChudnovskySeymour14}, where Chudnovsky and Seymour proved that Rao's poset is in fact a WQO.

We will show in the next section that graphs in $HCU$ can be characterized as those graphs whose degree sequences avoid certain ``minimal obstructions'' in the Rao poset. To that end, we develop a few results here about Rao-containment.

Our first result shows that we can recognize Rao-containment without having to examine the realizations of the degree sequences involved. Given a degree sequence $\pi$ of length $n$ and an integer $k \in \{0,\dots,n\}$, let us say that to \emph{augment $\pi$ by the term $k$} means to create a new degree sequence by increasing $k$ distinct terms of $\pi$ by 1 and inserting a new term equal to $k$. Because of the choices available in which terms of $\pi$ to increase, there are often several ways to augment $\pi$ by an integer $k$.

For example, augmenting the degree sequence $\pi=(3,2,2,1)$ by the terms $0$, $1$, $2$, $3$, or $4$ can yield any of the following degree sequences:
\begin{center}
    \begin{tabular}{cccccc}
        Augmenting by: &
        $0$ & $1$ & $2$ & $3$ & $4$ \\ \hline
         &
        $(3,2,2,1,0)$ & 
        $(4,2,2,1,1)$ & 
        $(4,3,2,2,1)$ & 
        $(4,3,3,3,1)$ &
        $(4,4,3,3,2)$\\
         &
         & 
        $(3,3,2,1,1)$ &
        $(4,2,2,2,2)$ & 
        $(4,3,3,2,2)$ &
         \\
         &
         &
        $(3,2,2,2,1)$ &
        $(3,3,3,2,1)$ & 
        $(3,3,3,3,2)$ &
         \\
         &
         &
         &
        $(3,3,2,2,2)$ & 
         &
    \end{tabular}
\end{center}

The act of augmenting a degree sequence by an integer $k$ mimics the effect of adding a new vertex of degree $k$ to a graph realizing $\pi$. We can similarly define an operation that mimics vertex deletion. To \emph{decrement $\pi$ by the term $k$} will mean to remove a term of $\pi$ equal to $k$ and to decrease $k$ of the remaining positive terms by 1. This requires both that $\pi$ have $k$ as one of its terms and that $\pi$ have at least $k+1$ positive terms to begin with. Note that each degree sequence in the lists above can be decremented by the number heading its column in such a way that $(3,2,2,1)$ is the result. We call the resulting sequences augmentations and decrementations of $\pi$, respectively. 

\begin{lemma} \label{lem: allowing RaoUp, RaoDown}
Let $\pi$ and $\rho$ be degree sequences of graphs. The following are equivalent.
\begin{enumerate}
    \item[\textup{(1)}] The relation $\pi \preceq \rho$ holds.
    \item[\textup{(2)}] The sequence $\rho$ can be formed by iteratively augmenting $\pi$ and each resulting sequence by non-negative integers.
    \item[\textup{(3)}] The sequence $\pi$ can be formed by iteratively decrementing $\rho$ and each resulting sequence by a term from the sequence.
\end{enumerate}
\end{lemma}
\begin{proof}
    To see that (1) implies (3), let $G$ and $H$ be realizations of $\pi$ and $\rho$, respectively, such that $G$ is an induced subgraph of $H$. Denote the vertices in $V(H) - V(G)$ by $v_1,\dots,v_p$. Letting $H_i = H_{i-1} -v_i$ for $1\leq i \leq p$ yields a sequence $H=H_0,H_1,\dots,H_p = G$ of induced subgraphs of $H$. Here the degree sequence of each $H_i$ is obtained by decrementing the degree sequence of $H_{i-1}$ by the degree (in $H_{i-1}$) of $v_i$. 
    
    The equivalence of (2) and (3) can be verified by paying attention to the inserted/removed term and the increased/decreased terms at each iterative step.
    
    To see that (2) implies (1), let $\pi=\pi^0,\pi^1,\dots,\pi^q= \rho$ be the degree sequences obtained at each stage of the augmenting that creates $\rho$ from $\pi$, and for $i \in \{0,\dots, q-1\}$ let $k_i$ denote the integer used in augmenting $\pi^i$ to form $\pi^{i+1}$. Now let $G_0$ be any realization of $\pi=\pi^0$. Form a realization $G_1$ of $\pi^1$ by adding a new vertex $v$ to $G_0$ and joining it by edges to $k_0$ vertices having degrees in $G_0$ that exactly match the terms of $\pi^0$ augmented to form $\pi^1$. The resulting realization $G_1$ contains $G_0$ as an induced subgraph. Proceeding inductively, we may construct $G_2,\dots,G_q = H$ in this same way. Since $H$ will be a realization of $\rho$ containing $G$ as an induced subgraph, we have $\pi \preceq \rho$.
\end{proof}

Lemma~\ref{lem: allowing RaoUp, RaoDown} allows us to more conveniently automate the process of determining whether $\pi \preceq \rho$ for given degree sequences $\pi,\rho$. 

We now show that the notion of Rao-containment allows us to characterize the graphs in $HCU$, along with those in $HU$ from~\cite{Ba12,Ba13}.

\begin{lemma} \label{lem: Rao chars exist}
    Suppose that $\pi \preceq \rho$. If $\rho$ is the degree sequence of a graph in $HU$, then so is $\pi$; likewise, if $\rho$ is the degree sequence of a graph in $HCU$, then so is $\pi$.
\end{lemma}
\begin{proof}
    Suppose first that $\pi \preceq \rho$ and that $\rho$ is the degree sequence of a graph $H \in HU$. Since $H$ is a unigraph, it is the only realization of $\rho$ up to isomorphism. By the definition of Rao-containment, there exists a realization $G$ of $\pi$ that is an induced subgraph of $H$. We know that $G$ and all its induced subgraphs are unigraphs because $H \in HU$; hence $\pi$ is the degree sequence of a graph $G \in HU$.
    
    Suppose instead that $\pi \preceq \rho$ and that $\rho$ is the degree sequence of a graph in $HCU$. Since $\pi \preceq \rho$, there exists a realization $G$ of $\pi$ that is an induced subgraph of a realization $H$ of $\rho$. Since the degree sequence of $H$ is that of a graph in HCU, by Lemma~\ref{prop:HCU_closed_degree} we conclude that $H$ is an induced subgraph of a unigraph $J$. Since the induced subgraph relation is transitive, $G$ is an induced subgraph of $J$, and $\pi$ is the degree sequence of a graph in $HCU$.
\end{proof}

We will present the characterization of the graphs in $HCU$ via Rao-containment in the next section. We use this later in Section 7 to find a forbidden induced subgraph characterization, so we conclude this section by developing the relationship between Rao-containment in degree sequences and containment of induced subgraphs in graphs.

Given a collection $\mathcal{D}$ of degree sequences, let $\mathcal{U}(\mathcal{D})$ denote the set of all degree sequences $\rho$ such that $\rho \succeq \pi$ for some $\pi \in \mathcal{D}$. Now consider the infinitely many graphs having an element of $\mathcal{U}(\mathcal{D})$ as their degree sequence. Many of these graphs are induced subgraphs of others. Let $\mathcal{M}(\mathcal{D})$ denote the minimal such graphs with respect to the induced subgraph order. Alternatively, we recognize that $G \in \mathcal{M}(\mathcal{D})$ if the degree sequence of $G$ belongs to $\mathcal{U}(\mathcal{D})$ but no proper induced subgraph of $G$ has its degree sequence in $\mathcal{U}(\mathcal{D})$. 

For example, suppose that $\mathcal{D} = \{(2,2,2)\}$. Then $\mathcal{M}(\mathcal{D})$ contains the graphs $K_3$, $P_5$, and $K_{2,3}$, among others. Note that neither $P_5$ nor $K_{2,3}$ contains $K_3$ as an induced subgraph, but each has a degree sequence that Rao-contains $(2,2,2)$, and no proper induced subgraph of these graphs has that same property. 

\begin{lemma} \label{lem: equivalent Rao and forb subgr}
    Let $\mathcal{D}$ be a collection of degree sequences. For any graph $G$, the following are equivalent.
    \begin{enumerate}
        \item[\textup{(1)}] The degree sequence of $G$ Rao-contains some element of $\mathcal{D}$.
        \item[\textup{(2)}] The graph $G$ contains an element of $\mathcal{M}(\mathcal{D})$ as an induced subgraph.
    \end{enumerate}
\end{lemma}
\begin{proof}
    Given a fixed collection $\mathcal{D}$ of degree sequences, in the following let $G$ be a graph, and let $\rho$ be its degree sequence.
    
    Suppose first that $\rho$ Rao-contains some element of $\mathcal{D}$. Then $\rho \in \mathcal{U}(d)$, and by definition, $G$ contains an element of $\mathcal{M}(\mathcal{D})$ as an induced subgraph. This shows that (1) implies (2).
    
    To prove the converse, suppose that $G$ contains some element $H$ of $\mathcal{M}(\mathcal{D})$ as an induced subgraph. If $\pi$ denotes the degree sequence of $H$, then we have $\rho \succeq \pi$. Now let $\tau$ denote a degree sequence in $\mathcal{D}$ such that $\pi \succeq \tau$. By the transitivity of Rao's relation, we see that $\rho \succeq \tau$. This shows that (2) implies (1) and completes our proof.
\end{proof}

We illustrate Lemma~\ref{lem: equivalent Rao and forb subgr} with a comment on the family $HU$ of hereditary unigraphs. In~\cite{Ba12}, Barrus presented a list $\mathcal{F}$ of 16 minimal forbidden induced subgraphs for $HU$; the elements of $\mathcal{F}$ are displayed in Figure~\ref{fig: hereditary unigraphs forbidden}. We add a new characterization of the class $HU$ in terms of Rao-containment.

\begin{figure}
    \centering
    \includegraphics[width=\textwidth]{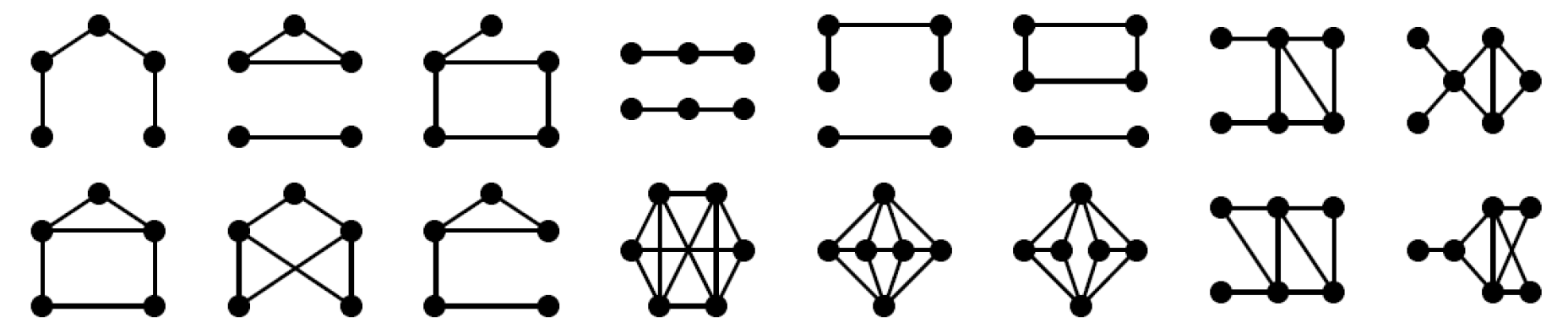}
    \caption{Minimal forbidden induced subgraphs for the family $HU$}
    \label{fig: hereditary unigraphs forbidden}
\end{figure}

\begin{theorem} \label{thm: hereditary unigraphs Rao}
The following are equivalent for a graph $G$ having degree sequence $\pi$.
\begin{enumerate}
    \item[\textup{(1)}] The graph $G$ belongs to $HU$.
    \item[\textup{(2)}] The graph $G$ contains no element of $\mathcal{F}$ as an induced subgraph.
    \item[\textup{(3)}] The sequence $\pi$ does not Rao-contain any element of $\mathcal{L}_1$, where \begin{multline} \mathcal{L}_1 = \{(2,2,2,1,1), (3,2,2,2,1), (3,3,2,2,2), \nonumber \\ (2,2,1,1,1,1), (4,3,3,2,1,1), (4,4,3,2,2,1), (4,4,4,4,3,3)\}. \nonumber \end{multline}
    \end{enumerate}
\end{theorem}
\begin{proof}
The equivalence of (1) and (2) was shown in~\cite{Ba12}. Observe that no degree sequence in $\mathcal{L}_1$ is unigraphic because at least two realizations of each appear in Figure~\ref{fig: hereditary unigraphs forbidden}. Hence by Lemma~\ref{lem: Rao chars exist} the statement (1) implies (3).

We now show that (3) implies (2). Suppose that a graph $G$ having degree sequence $\pi$ contains an element $F$ of $\mathcal{F}$ as an induced subgraph. Then $\pi$ Rao-contains the degree sequence of $F$ by definition. If $F$ is the graph $K_2+C_4$ or its complement, then $\pi$ Rao-contains $(2,2,2,2,1,1)$ or $(4,4,3,3,3,3)$, which respectively Rao-contain the sequences $(2,2,2,1,1)$ and $(3,3,2,2,2)$ from $\mathcal{L}_1$. Otherwise, the degree sequence of $F$ is listed in $\mathcal{L}_1$, so $\pi$ Rao-contains an element of $\mathcal{L}_1$.
\end{proof}

\section{A Rao-minimal degree sequence characterization of $HCU$} \label{sec: Rao-Minimal}

Lemma~\ref{lem: Rao chars exist} implies that graphs in $HU$ and $HCU$ can be characterized as those whose degree sequences do not Rao-contain certain minimal ``forbidden'' sequences. This is illustrated for the family $HU$ in part (3) of Theorem~\ref{thm: hereditary unigraphs Rao}, where we give a list $\mathcal{L}_1$ of seven sequences that are not Rao-contained by the degree sequence of any graph in $HU$. Conversely, any graph not in $HU$ has a degree sequence that Rao-contains at least one element of $\mathcal{L}_1$.

In this section we characterize graphs in $HCU$ in a similar manner. We determine a minimal list $\mathcal{L}_2$ of degree sequences such that a graph $G$ belongs to $HCU$ if and only if the degree sequence of $G$ does not Rao-contain any element of $\mathcal{L}_2$. We will do this in two parts, handling the degree sequences of non-split and split graphs separately. The resulting elements of $\mathcal{L}_2$ are listed in Theorems~\ref{prop:Rao_minimal_nonsplit} and~\ref{prop:Rao_minimal_split}.

In light of Lemma~\ref{prop:HCU_closed_degree}, a degree sequence is \emph{forbidden} if it is not the degree sequence of a graph in $HCU$. A forbidden degree sequence is \emph{Rao-minimal} if it does not Rao-contain any other forbidden degree sequence. (Note that the degree sequences comprising $\mathcal{L}_1$ in Theorem \ref{thm: hereditary unigraphs Rao}(3) are each Rao-minimal with respect to the corresponding notion of ``forbidden'' for $HU$.) From the definition of Rao-containment, it follows that Rao-minimal sequences are related to minimal forbidden induced subgraphs. We record this, along with several observations derived from Corollary \ref{prop:MFIS_indecomp_closed}, as a remark. 

\begin{remark}
\label{prop:Rao_min_indecomp_complement}
    \rm A forbidden degree sequence $\pi$ is Rao-minimal if and only if its realizations are all minimal forbidden induced subgraphs for the class $HCU$. Moreover, each Rao-minimal sequence $\pi$ is indecomposable, and the degree sequence of the complement, and, if split, the inverse and inverse-complement, of any realization of $\pi$ is also Rao-minimal. 
\end{remark}

\begin{figure}
    \centering
    \includegraphics[height=4cm]{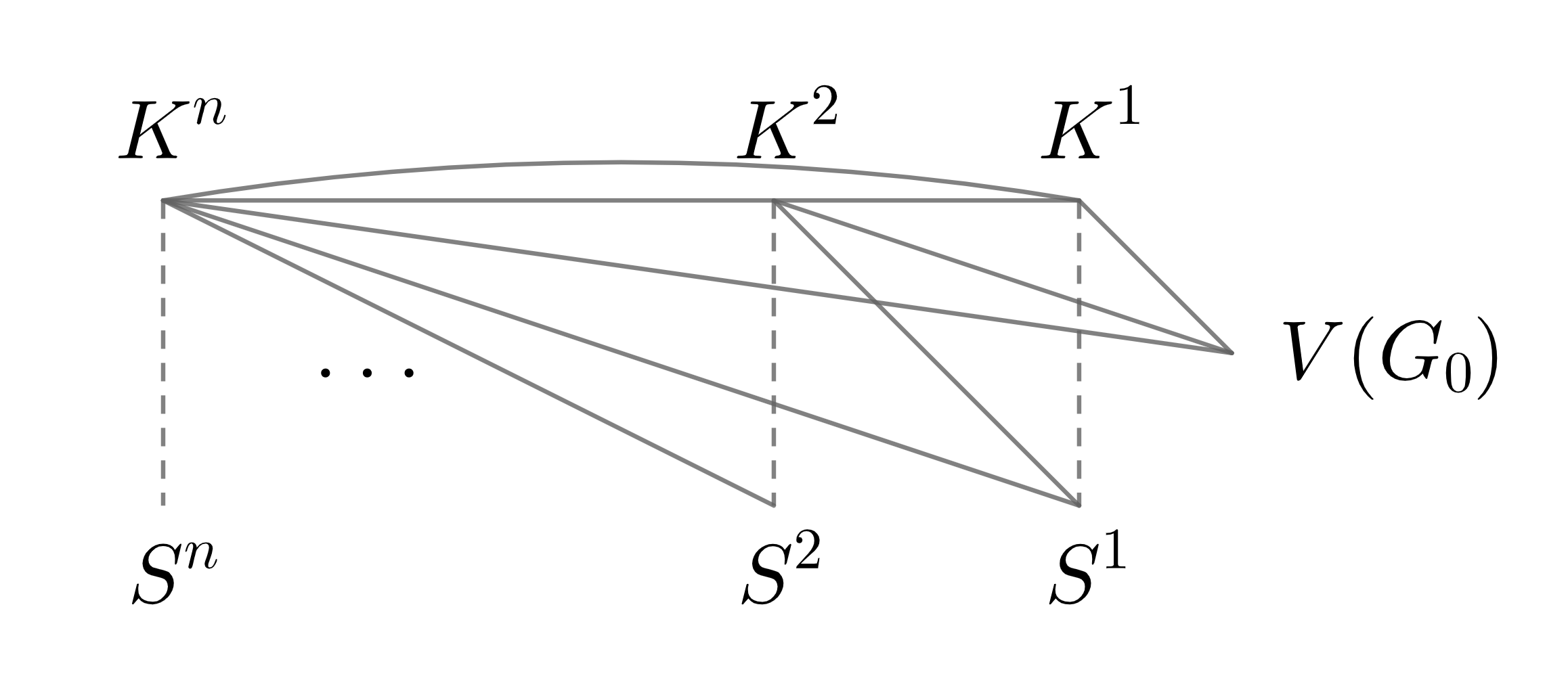} 
    \caption{The structure of the composition $G_n \circ \ldots \circ G_0$, where $V(G_i)$ is given the $KS$-partition $K_i \cup S_i$. Solid lines indicate that all edges are present between two vertex sets, dotted lines indicate that some edges may exist, and two sets not connected with a line have no edges between them.}
    \label{fig:composition_degree}
\end{figure}

Our approach will often turn to decrementations of  degree sequences. The idea behind these decrementations arises from a result of Kleitman and Wang.

\begin{theorem}[{\cite[Theorem 2.1]{KlWa73}}] \label{prop: KW theorem}
Given a graphic sequence $(d_1,\dots,d_n)$, for any $k \in \{1,\dots,n\}$ there exists a graph having vertices $\{v_1,\dots,v_n\}$ in which each vertex $v_i$ has degree $d_i$, and the vertex $v_k$ is adjacent to the first $d_k$ vertices other than itself.
\end{theorem}

The result of Kleitman and Wang implies that we may form another graphic sequence from $(d_1, \ldots, d_n)$ by imagining the impact of deleting a vertex $v_k$ whose neighbors are as described in the theorem.

\begin{definition}
\label{def:KW_reduction}
    {\rm Let $\pi$ be a degree sequence with $\pi = (d_1, \ldots, d_n)$. For $1 \le i \le n$, the} KW-reduction at index $i$ {\rm is the degree sequence $\pi'$ obtained by deleting $d_i$ from $\pi$ and decrementing $d_i$ remaining largest terms by $1$.}
\end{definition}

For example, given $\pi = (4,3,2,2,2,1)$, the KW-reduction at index $6$ is $(3,3,2,2,2)$, and at index $2$ is $(3, 2, 1, 1, 1)$. Since $d_4 = d_5$, we finished the KW-reduction at index $2$ by reordering and reindexing the terms.  

\begin{remark}
\label{prop: KW leads to Rao-contained}
    \rm The KW-reduction at index $i$ of a graphic sequence is a special case of decrementing a sequence by its $i$th term as described in the previous section. By Lemma~\ref{lem: allowing RaoUp, RaoDown} this implies that if $\tau$ is a KW-reduction of $\pi$ at any index, then $\pi$ Rao-contains $\tau$.
\end{remark}

It will be helpful to us to examine KW-reductions of degree sequences in light of Tyshkevich decomposition. The following analysis slightly generalizes the ideas in \cite{Ba12b}, which dealt with the Havel-Hakimi algorithm. 

Given degree sequences $\pi^0, \ldots, \pi^k$ and a sequence $\pi$ that can be written as $\pi^k \circ \dots \circ \pi^1 \circ \pi^0$, let $\pi'$ be the KW-reduction of $\pi$ at an index in $\pi^j$. It is straightforward to see that \[\pi' = \pi^k \circ \dots \circ \pi^{j+1} \circ \rho \circ \pi^{j-1} \circ \dots \circ \pi^0,\] where $\rho$ is the corresponding KW-reduction on $\pi^j$. Note that no assumption is made that $\pi^0, \ldots, \pi^k$ are indecomposable. 

Indeed, whether a degree sequence is decomposable may change (in either direction) between a degree sequence and its KW-reduction. For example, if a graph $G$ contains a single dominating or isolated vertex, and the decomposition of its degree sequence $\pi$ can be written as $(0) \circ \rho$, where $\rho$ is indecomposable, then the KW-reduction at index $1$ of the decomposable sequence $\pi$ is the indecomposable sequence $\rho$. Likewise, a KW-reduction on an indecomposable sequence may be a decomposable sequence. This arises exactly from dominating, isolated, and swing vertices. 
		
\begin{theorem}
\label{prop:reduction_decomposable}
	If $\pi$ is an indecomposable sequence of length $n$, and $\pi'$ is a \emph{decomposable} sequence obtained as a KW-reduction of $\pi$, then $\pi'$ may be written as $(0) \circ \rho$ or $\rho \circ (0)$ for some graphic sequence $\rho$. 
\end{theorem}
\begin{proof}
	Let $\pi$ be an indecomposable sequence of length $n$ and let $\pi'$ be the KW-reduction of $\pi$ indexed at a term with value $t$. Let $G$ be a realization of $\pi$ having a vertex $v$ of degree $t$ with the property (guaranteed by Theorem~\ref{prop: KW theorem}) that the neighbors of $v$ have degrees as high as possible. Thus $\pi'$ is the degree sequence of $G'=G-v$, i.e., deleting $v$ from $G$ has the same effect on the degree sequence of $G$ that a KW-reduction does changing $\pi$ into $\pi'$. Suppose that $\pi'$ cannot be written as $(0) \circ \rho$ for any sequence $\rho$. We show $\pi$ can be written as $\rho \circ (0)$ for some sequence $\rho$. 
			 
	Our proof proceeds by showing that if the leftmost indecomposable part in the Tyshkevich decomposition of a graph $G'$ has more than one vertex, then restoring vertex $v$ to $G'$ (forming the graph $G$) so as to make $G$ indecomposable requires that $G'$ have a single vertex in the rightmost component of its Tyshkevich decomposition. Throughout, we notate the induced subgraph of $G$ on a set $V$ of vertices by $G[V]$ and let $\deg_G V$ denote the set $\{\deg_G (v) | v \in V\}$. For any natural number $k$, we write $\deg_G V \le k$ (respectively, $<$) if for all $v \in V$, $\deg_G (v) \le k$ (resp., $<$). 
	
	Since $G'$ is decomposable, it can be written as $(G'[A \cup B],A,B) \circ G'[C]$, where $G'[A\cup B]$ is indecomposable and has more than one vertex, and $G'[C]$ has at least one vertex.  (As in Definition \ref{def:composition}, we are specifying the $KS$-partition of $G'[A \cup B]$). This forces both $A$ and $B$ to be nonempty. Now since $G$ is indecomposable, the vertex $v$ must have a neighbor in $B$ or a non-neighbor in $A$ (otherwise, $G$ could be written as $(G[A \cup B],A,B)\circ G[C \cup \{v\}]$ for a contradiction). Suppose first that $v$ has a neighbor in $B$. Let $b$ be a neighbor of $v$ in $B$.
	
	For any $w \in V(G)$ with $\deg_{G'}(w) \ge \deg_{G'}(b) + 2$, we have $\deg_G(w) > \deg_G(b)$, and since the deletion of $V$ of $G$ corresponds to a KW-reduction on the degree sequence of $G$, we conclude that $vw \in E(G)$. Thus $v$ is adjacent to each vertex whose degree in $G'$ is at least two greater than $\deg_{G'}(b)$. Now $G'[A \cup B]$ is indecomposable, so each vertex of $A$ has a neighbor and a non-neighbor in $G$ and vice-versa. This means $\deg_{G'}(b) \le |A|-1$. Any $a \in A$ has a neighbor in $B$ and is adjacent to every vertex in $A$ and $C$, excepting itself, so $\deg_{G'}(a) \ge |A|+|C| \ge \deg_{G'}(b) + 2$. It follows that $v$ is adjacent to every vertex of $A$. 

	We now consider $v$'s adjacencies in $C$. If $v$ is adjacent to all of $C$, then we may write $G$ as $(G[A \cup B \cup\{v\}],A\cup\{v\},B) \circ G[C]$, contradicting the indecomposability of $G$. Hence $v$ has at least one non-neighbor in $C$; let $C'$ denote the set of such vertices. Since $\deg_{G'}(b) \le |A|-1$, if any vertex in $C'$ has any adjacency in $C$, then its degree is at least $|A|+1$, thus at least two greater than $\deg_{G'}(b)$, and thus the vertex in $C'$ would be adjacent to $v$, for a contradiction. Hence each vertex of $C'$ is isolated in $G[C]$. At this point we may express $G$ as $(G[A \cup B \cup C' \cup \{v\}], A \cup \{v\}, B \cup C') \circ G[C\setminus C']$. Since $G$ is indecomposable, we conclude that $C=C'$. Therefore $G[C]$ is a set of isolated vertices, and the rightmost component of the Tyshkevich decomposition of $G$ is trivial. 
			
    Now, suppose $v$ has no neighbor in $B$, and instead must have a non-neighbor $a$ in $A$. It is straightforward to verify in this case that the complements of $G$ and $G'$, along with their connection via a KW-reduction on the degree sequence of the complement of $\pi$ indexed at a term equal to $n-1-t$, satisfy the conditions of the theorem. The result follows because swing vertices are preserved by KW-reductions and complementation. 
\end{proof}

\subsection{Non-split Rao-minimal sequences} \label{subsec: nonsplit minimal}

As above, $\mathcal{L}_2$ denotes the complete list of Rao-minimal forbidden degree sequences for $HCU$. We characterize the non-split Rao-minimal sequences first.

\begin{theorem}
\label{prop:Rao_minimal_nonsplit}
The degree sequences
\[(3,3,2,2,2), \ (2,2,2,1,1), \ (2,2,1,1,1,1), \  (4,4,4,4,3,3), \ (5,4,3,3,3,1,1), \text{ and } \  (5,5,3,3,3,2,1)\] are elements of $\mathcal{L}_2$; these are exactly the Rao-minimal forbidden degree sequences of non-split graphs. 
\end{theorem}

\begin{proof}

In this proof, we call a degree sequence a \emph{target sequence} if it is indecomposable, forbidden, and non-split. By Lemma~\ref{prop:EG_nonsplit_useful} and Theorem~\ref{prop:degree_characterization_nonsplit}, the degree sequences listed in the theorem are target sequences. By Remark \ref{prop:Rao_min_indecomp_complement}, non-split Rao-minimal sequences must be target sequences, so it suffices to prove that of all target sequences, exactly those listed in Theorem \ref{prop:Rao_minimal_nonsplit} are Rao-minimal. The graphs $C_4$ and $2K_2$ are the only two non-split graphs on four or fewer vertices, and both are included in $HCU$. This implies that $(3,3,2,2,2)$ and $(2,2,2,1,1)$ are Rao-minimal. The other four sequences listed in the theorem do not Rao-contain any smaller target sequences, so they are also Rao-minimal. 

We proceed inductively, showing that for any target sequence $\pi$ on six or more vertices not listed in Theorem \ref{prop:Rao_minimal_nonsplit}, there exists a degree sequence $\pi'$ on one or two fewer vertices such that $\pi'$ is a target sequence and some realization of $\pi'$ is induced in some realization of $\pi$. Hence $\pi$ is not Rao-minimal.

Let $\pi$ be a target sequence of length $n \ge 6$ other than those listed in Theorem \ref{prop:Rao_minimal_nonsplit}. First suppose that there exists $i \in \{1, \ldots, n\}$ for which the KW-reduction at index $i$ satisfies none of the five conditions in Theorem \ref{prop:degree_characterization_nonsplit} and contains no terms equal to $0$ or $n-2$. Let $\pi'$ be this KW-reduction. By Remark~\ref{prop: KW leads to Rao-contained}, $\pi'$ is Rao-contained by $\pi$, and by Theorem \ref{prop:reduction_decomposable}, $\pi'$ is indecomposable. 

Suppose for a contradiction that there exists an (indecomposable) split graph $G'$ realizing $\pi'$, with $KS$-partition $K' \cup S'$. Let $G$ be the graph created from $G'$ by adding a vertex $v$ adjacent to $d_i$ vertices $v_1, \ldots, v_{d_i}$ of highest degree in $G'$. Since $\pi'$ is a KW-reduction of $\pi$ at index $i$, it follows that $G$ is a realization of $\pi$. In an indecomposable split graph, the clique vertices have greater degree than the stable set vertices. If $d_i \le |K'|$, then $v$ is adjacent only to clique vertices, so $G$ is a split graph with clique $K'$ and stable set $S' \cup \{v\}$. If $d_i > |K'|$, then $v$ is adjacent to all clique vertices and some stable set vertices, so $G$ is a split graph with clique $K' \cup \{v\}$ and stable set $S'$. Either way, we conclude that $G$ is split. 

Hence $\pi'$ is a target sequence Rao-contained by $\pi$, and we conclude that $\pi$ is not Rao-minimal. Therefore, as claimed, if for some $i \in \{1,\dots,n\}$ the KW-reduction $\pi'$ at index $i$ has no terms equal to $0$ or $n-2$ and satisfies none of the five conditions in Theorem \ref{prop:degree_characterization_nonsplit}, then $\pi$ is not Rao-minimal.

We now consider the remaining possibility, i.e., what occurs if for all $i$, $1 \le i \le n$, the KW-reduction of $\pi$ at index $i$ satisfies one of the five conditions in Theorem \ref{prop:degree_characterization_nonsplit} or contains a term equal to $0$ or $n-2$. We will show that $\pi$ Rao-contains an element of the list $\mathcal{L}_2$ in the statement of Theorem~\ref{prop:Rao_minimal_nonsplit}.

Our assumption about the KW-reductions certainly pertains to the KW-reduction at index $2$, and we begin our analysis here. Since $\pi$ is indecomposable, $1 \leq d_j \leq n-2$ for all $j \in \{1,\dots,n\}$. Let $\pi'$ be the KW-reduction at index $2$, and let $x=d_2$ denote the removed term. If $x = 1$, then $\pi$ satisfies condition (1) of Theorem~\ref{prop:degree_characterization_nonsplit}, a contradiction; hence $x \geq 2$, and the last term of $\pi$ to be reduced by 1 is always equal to $d_{x+1}$ (because of the removal of the term $d_2$ in the KW-reduction). Recall that $n \geq 6$ and that the terms of $\pi$ are assumed to be in non-increasing order.

By the present assumptions, $\pi'$ satisfies one of the conditions (1)-(5) from Theorem~\ref{prop:degree_characterization_nonsplit}, or $\pi'$ has a term equal to $0$ or $n-2$. We claim that statements (i) through (vii) below convert these possibilities into conditions on the sequence $\pi$. 

\medskip
\begin{paracol}{2}
    \begin{enumerate}
        \item [\rm{(i-a)}] 
        $d_3 = \cdots = d_{x+1}=2$ and $d_{x+2}= \cdots = d_n = 1$.
        \item [\rm{(i-b)}] 
        $d_1 = \cdots = d_4 = 2$ and $d_5 = \cdots = d_n = 1$.
        \item [\rm{(ii-a)}]  
        $d_1 = \cdots = d_{x+1} = n-2$ and $d_{x+2} = \cdots = d_{n-1} = n-3$.
        \item [\rm{(ii-b)}]
        $d_1 = \cdots = d_{n-2} = n-2$ and $d_{n-1} = d_n = n-3$.
        \item [\rm{(iii-a)}] 
        $d_1 = n-2, d_3 = 3, $ and $d_{x+2} \le 2$. 
        
        \switchcolumn
        \item [\rm{(iii-b)}] 
        $d_1 = n-2, d_2 \ge 3$ and $d_3 = \cdots = d_{x+2} = 2$.
        \item [\rm{(iv)}] 
        $d_2 = \cdots = d_{n-2} = n-3$, $d_{n-1} = n-4$, and $d_n = 1$.
        \item[\rm{(v)}] $\pi = (3,3,3,3,2,2)$
        \item [\rm{(vi)}] 
        $d_{x+1} = \cdots = d_n =1$.
        \item [\rm{(vii)}] 
        $d_1 = \cdots = d_{x+2}=n-2$.
    \end{enumerate}
    \end{paracol}

\medskip
It is straightforward to verify that conditions (vi) and (vii) are respectively equivalent to the sequence $\pi'$ containing a term equal to $0$ or to $n-2$. We further claim that each of the conditions (1)--(5) from Theorem~\ref{prop:degree_characterization_nonsplit} implies at least one of the conditions (i-a)--(vii), most often one having a corresponding Roman numeral. Conversely, we claim that (1)--(5) are each implied by the statement or disjunction of statements listed above having the same Roman numeral.

We prove a representative case and leave verification of the others to the reader. Observe that if $\pi$ satisfies (i-a) or (i-b), then $\pi'$ satisfies condition (1). Conversely, suppose that $\pi'$ satisfies (1). Since forming $\pi'$ involves removing the second term of $\pi$, the second term of $\pi'$ is equal to \[\left\{\begin{matrix*}[l] d_3-1 & \text{ if } \ d_3>d_{x+2} & \text{Example: \ } \pi = (2,2,2,1,1);\\
d_3 & \text{ if } \ d_1>d_3 = d_{x+2} & \text{Example: \ }\pi = (2,2,1,1,1,1);\\
d_3-1 & \text{ if } \ d_1=d_3 =d_{x+2}>d_{x+3} & \text{Example: \ } \pi = (2,2,2,2,1,1);\\
d_3 & \text{ if } \ d_1=d_3 = d_{x+2}=d_{x+3} & \text{Example: \ } \pi = (1,1,1,1,1,1).
\end{matrix*}\right.\] Condition (1) requires that the second term of $\pi'$ be 1, so in the first of the cases above, $d_3=2$ and $d_{x+2}=1$, yielding (i-a) or (vi). In the second case, $d_3=\dots=d_n=1$, so (vi) holds, and in the third case, $d_1=x=d_3=2$, and (i-b) holds. The fourth case leads to a contradiction, since then $d_2=1$ and $\pi$ is not a target sequence by Theorem~\ref{prop:degree_characterization_nonsplit}. The other verifications are similar.

Hence $\pi$ satisfies one of the ten conditions (i-a) through (vii) listed above, and we will consider each one. In most cases a condition from this list quickly leads to a contradiction or to the conclusion that $\pi$ Rao-contains an element of the list $\mathcal{L}_2$.

If $\pi$ satisfies (i-a), then it has the form $(d_1,x,2,...,2,1,...1)$, where there are $x-1$ terms equal to $2$ following $x$. We show by induction that all such sequences of length at least 5 Rao-contain $(2,2,2,1,1)$. When the length is 5, parity considerations show that $(2,2,2,1,1)$ is the only sequence satisfying (i-a). Now assume that $n>5$ and that every target sequence of length at most $n-1$ that satisfies (i-a) Rao-contains $(2,2,2,1,1)$. If $d_1=2$, then $\pi$ clearly Rao-contains $(2,2,2,1,1)$. If $d_1>2$ and $x=2$, then $\pi$ has multiple terms equal to 1, and the KW-reduction at index $n$ on $\pi$ does not change any term from $\pi$ but the first term and the omitted last term; then this KW-reduction satisfies (i-a), and by the induction hypothesis and the transitivity of Rao-containment, $\pi$ Rao-contains $(2,2,2,1,1)$. Finally, if $x>2$, then the fourth term of $\pi$ equals 2; here the KW-reduction at index $3$ reduces the first two terms of $\pi$ by 1 and leaves later remaining terms unchanged, so this KW-reduction again satisfies (i-a), and we conclude that $\pi$ Rao-contains $(2,2,2,1,1)$.

If $\pi$ satisfies (i-b), then it has the form $(2,2,2,2,1,\dots,1)$, where we may assume that there are at least two 1 terms. It is straightforward to find a realization of $\pi$ that contains $P_5$ as an induced subgraph, so $\pi$ Rao-contains $(2,2,2,1,1)$.

The sequence $\pi$ cannot satisfy (ii-a): here $x=n-2$, forcing $d_{n-1}$ to equal both $n-2$ and $n-3$. 

Suppose that $\pi$ satisfies (ii-b). Note that $n$ must be even, and $\pi$ has as one of its realizations the complement of $P_4+(n/2 - 2) K_2$, which has the complement of $P_4+K_2$ as an induced subgraph. It follows that $\pi$ Rao-contains $(4,4,4,4,3,3)$. 

We will return to conditions (iii-a) and (iii-b) later. We next show that if any target sequence of length at least 5 satisfies (iv), then the sequence Rao-contains $(2,2,2,1,1)$. We use induction on the length, noting that $(2,2,2,1,1)$ itself satisfies (iv). Fix $n \geq 6$, and assume that this claim holds for all sequences satisfying (iv) with length between $5$ and $n-1$. If $n$ is even, then parity requires that $d_1=n-2$, and performing the KW-reduction at index $1$ on $\pi$ results in a sequence of smaller length that also satisfies (iv) and hence Rao-contains $(2,2,2,1,1)$. If $n$ is odd, then $d_1=n-3$. Consider the degree sequence $\rho$ of length $n-1$ satisfying (iv). A realization of $\pi$ may be formed by adding a vertex to a realization $G'$ of $\rho$ that is adjacent exactly to all vertices of $G'$ except for the vertex of maximum degree and the vertex of minimum degree. It follows that $\pi$ Rao-contains $\rho$, and by transitivity and the induction hypothesis, $\pi$ Rao-contains $(2,2,2,1,1)$. 

If $\pi$ satisfies (v) - that is, $\pi = (3,3,3,3,2,2)$ - then $\pi$ Rao-contains $(3,3,2,2,2)$, and our conclusion holds. 

We also delay handling condition (vi) momentarily. If $\pi$ satisfies (vii), then $x=n-2$ and in fact every term of $\pi$ is $n-2$. This is a contradiction, since by condition (2) of Theorem~\ref{prop:degree_characterization_nonsplit}, $\pi$ would not be forbidden.

It remains to show that our target sequence $\pi$ Rao-contains an element of $\mathcal{L}_2$ when $\pi$ satisfies one of the conditions (iii-a), (iii-b), or (vi). To assist us, we consider the ``complementary'' degree sequence $\overline{\pi} = (n-1-d_n, n-1-d_{n-1},\dots,n-1-d_1)$ with $i$th term $n-1-d_{n+1-i}$. Since the graph properties of being indecomposable, forbidden, and non-split are all preserved under graph complementation, the arguments thus far in the proof also apply to $\overline{\pi}$. If the forbidden sequence $\overline{\pi}$ Rao-contains an element of $\mathcal{L}_2$, then the same must be true for $\pi$ by Remark~\ref{prop:Rao_min_indecomp_complement}. Hence we may assume that $\overline{\pi}$ satisfies a requirement analogous to conditions (iii-a), (iii-b), or (vi). This implies that $\pi$ satisfies one of the additional conditions listed below, where $y=d_{n-1}$. Note that the second term $\overline{x}$ of $\overline{\pi}$ equals $n-1-y$ (so for instance the term of $\overline{\pi}$ with index $\overline{x}+2$ is equal to $n-1-d_{n+1-(\overline{x}+2)} = n-1-d_{n+1-(n-1-y)-2} = n-1-d_y$). We may assume that $y \leq n-3$ since otherwise condition (2) of Theorem~\ref{prop:degree_characterization_nonsplit} would imply that $\pi$ is not forbidden.

\begin{enumerate}
    \item[(iii-a$'$)] 
    $d_n = 1$, $d_{n-2} = n-4$, and $d_{y} \geq n-3$;
    \item[(iii-b$'$)]  
    $d_n = 1$, $d_{n-1} \leq n-4$, and $d_{y}=\dots=d_{n-2} = n-3$;
    \item[(vi$'$)] 
    $d_{1} = \dots = d_{y+1} = n-2$.
\end{enumerate}

Suppose first that $\pi$ satisfies condition (iii-a). If $\pi$ also satisfies (iii-a$'$), then $3=d_3 \geq d_{n-2}=n-4$, so $n \leq 7$. The only non-increasing, graphic sequences of length 6 or 7 satisfying these conditions are $(4,3,3,2,1,1)$, $(4,4,3,2,2,1)$, $(5,4,3,3,3,1,1)$, and $(5,5,3,3,3,2,1)$. The first two are split sequences, contrary to our assumption on $\pi$. The last two are elements of $\mathcal{L}_2$. 

If instead $\pi$ satisfies (iii-a) and (iii-b$'$), then since $d_3=3 \leq n-3 = d_{n-2} \leq 3$, we have $n=6$. The only non-increasing, graphic sequence of length 6 satisfying these conditions is $(4,3,3,3,2,1)$, which Rao-contains $(2,2,2,1,1)$.

Suppose that $\pi$ satisfies (iii-a) and (vi$'$). If $y>1$ then $3=d_3 = n-2$ and $n=5$, contrary to our assumption. With $y=1$ we see that $d_1=d_2=n-2$ and $d_{n-1}=d_n = 1$. If $G$ is any realization of $\pi$, then each vertex of $G$ of degree $n-2$ must be adjacent to at least one vertex of degree 1; the only way for this to happen is if the two vertices with degrees $d_1$ and $d_2$ are both adjacent to all other vertices in the graph. This is a contradiction, for then $G$ and hence $\pi$ are decomposable in the Tyshkevich decomposition; we may write $\pi = (2,2,1,1) \circ \pi^*$, where $\pi^*$ is the degree sequence of the induced subgraph resulting from deleting two vertices of maximum degree and two of minimum degree from $G$. Hence $\pi$ cannot satisfy (iii-a) and (vi$'$).

If $\pi$ satisfies (iii-b) and (iii-b$'$), then a contradiction arises, as $2 = d_3 \geq d_{n-2}=n-3$. 

Suppose that $\pi$ satisfies (iii-b) and (vi$'$). Since $d_{y+1}= n-2 > 2=d_3$, we would need $y=1$ and hence $x=d_2=n-2$, but then $d_n = d_{x+2} = 2 > 1 = d_{n-1}$, a contradiction.

Finally, suppose that $\pi$ satisfies (vi) and (vi$'$). Here $x=n-2$, so $\pi$ satisfies $d_1=d_2=n-2$ and $d_{n-1}=d_n=1$. As before, this is a contradiction, since $\pi$ is then decomposable. 

Note now that if $\pi$ satisfies (iii-b) and (iii-a$'$), (vi) and (iii-a$'$), or (vi) and (iii-b$'$), then the arguments of the preceding paragraphs, applied appropriately to the degree sequence $\overline{\pi}$, either end in contradiction or show that $\overline{\pi}$ (and hence $\pi$) Rao-contains a degree sequence from $\mathcal{L}_2$. This completes the proof of Theorem~\ref{prop:Rao_minimal_nonsplit}.
\end{proof}

In the proof of the theorem, we show that every non-minimal target sequence can be reduced to another \textit{non-split} sequence, and none Rao-contain only a split Rao-minimal sequence. We record this as a corollary. 

\begin{corollary}
\label{prop:nonsplit_contains_nonsplit}
    Let $\pi$ be an indecomposable, forbidden, non-split degree sequence. Then $\pi$ Rao-contains a non-split element of $\mathcal{L}_2$. 
\end{corollary}

\subsection{Split Rao-minimal sequences}

We now characterize the Rao-minimal split degree sequences, which are all of length $8$. 

\begin{theorem}
	\label{prop:Rao_minimal_split}
	The degree sequences
	
	\smallskip
	\noindent
	\hfil \begin{tabular}{llll}
	 $(5, 5, 4, 4, 2, 2, 1, 1)$, & $(5, 4, 4, 2, 2, 1, 1, 1)$, & $ (5, 5, 4, 2, 2, 2, 1, 1)$, & $(5, 5, 5, 4, 2, 2, 2, 1)$, \\
	 $(5, 5, 5, 5, 2, 2, 2, 2)$, & $(6, 5, 4, 4, 3, 2, 1, 1)$, & $(6, 5, 5, 4, 3, 2, 2, 1)$, & $(6, 5, 5, 4, 3, 3, 1, 1)$, \\
	 $(6, 5, 5, 5, 3, 2, 2, 2)$, & $(6, 6, 4, 4, 3, 2, 2, 1)$, & $(6, 6, 5, 4,3, 3, 2, 1)$, & $(6, 6, 5, 5, 3, 3, 2, 2)$, \\
	 $(6, 6, 5, 5, 5, 3, 2, 2)$, & $(6, 6, 6, 5, 5, 3, 3, 2)$ & & 
	 \end{tabular} \hfil
 	
 	\smallskip
	\noindent are elements of $\mathcal{L}_2$; these are exactly the Rao-minimal forbidden degree
	 sequences of split graphs.
\end{theorem}

Our arguments proving Theorem~\ref{prop:Rao_minimal_split} will be similar to those in Section~\ref{subsec: nonsplit minimal}, though here we will replace degree sequences by cross-degree sequence pairs in order to take advantage of Theorem \ref{prop:degree_characterization_split}. Given two cross-degree sequence pairs $\crosspair$ and $\redpair$, we define $\crosspair$ to be \emph{Rao-contained by $\redpair$} if some realization of $\crosspair$ is an induced subgraph of some realization of $\redpair$. Equivalently, $\crosspair$ is Rao-contained by $\redpair$ if any degree sequence corresponding to $\crosspair$ is Rao-contained by some degree sequence corresponding to $\redpair$. A cross-degree sequence pair $\crosspair$ is \emph{forbidden} or \emph{Rao-minimal} if a degree sequence corresponding to $\crosspair$ is forbidden or Rao-minimal, respectively, and $\crosspair$ is a \emph{target pair} if any (and hence all) of its corresponding degree sequences is indecomposable and forbidden. Finally, if $\alpha=(a_1,\dots,a_n)$ and $\beta=(b_1,\dots,b_m)$, define the \emph{length} of $\crosspair$ as $n+m$, the total number of terms within the two sequences.

The degree sequences in Theorem \ref{prop:Rao_minimal_split} correspond to the following cross-degree sequence pairs. Since cross-degree sequence pairs are unordered, there exist three instances of pairs of degree sequences above mapping to the same cross-degree sequence pair. 
\begin{equation}
	\begin{matrix}
		\{(2,2,1,1,1),(3,2,2)\}, & \{(2,2,2,1,1),(3,3,2)\}, & \{(2,2,1,1),(2,2,1,1)\}, & \{(3,2,1,1),(3,2,1,1)\},\\
	\{(2,2,2,1),(2,2,2,1)\}, & \{(2,2,2,2),(2,2,2,2)\}, & \{(3,2,2,1),(3,2,2,1)\}, & \{(3,3,1,1),(3,2,2,1)\}, \\
	\{(3,2,2,2),(3,2,2,2)\}, & \{(3,3,2,1),(3,3,2,1)\}, & \{(3,3,2,2),(3,3,2,2)\}. &
	\end{matrix}\tag{$*$}\label{eq: split cross degree lists}
\end{equation}

\noindent It is not hard to show by Lemma \ref{prop:cross_degree_GaleRyser} and Theorem \ref{prop:degree_characterization_split} that each of these is a target pair.

In preparation for the proof of Theorem~\ref{prop:Rao_minimal_split}, we note when an unordered pair of sequences is the cross-degree sequence pair of a split graph. This lemma, due independently to Gale and Ryser, was originally proven in the context of bipartite graphs. 

\begin{lemma}[\cite{Ga57,Ry57}]
	\label{prop:graphic_Gale_Ryser}
	Let $\alpha = (a_1, \ldots, a_n)$ and $\beta = (b_1, \ldots, b_m)$ be two non-increasing, non-negative sequences. The unordered pair $\crosspair$ is the cross-degree sequence pair of a split graph if and only if $\crosspair$ satisfies the $k$th Gale--Ryser inequality given in Lemma \ref{prop:cross_degree_GaleRyser} for all $1 \le k \le n$. 
\end{lemma}

We also find an analogue of Theorem~\ref{prop: KW theorem} for cross-degree sequence pairs. Recall that in each of $\alpha$ and $\beta$ the terms are assumed to be in non-increasing order.

\begin{lemma}
	\label{prop:Rao_induced_split_KW}
	Let $\crosspair$ be a cross-degree sequence pair of a graph with $\alpha = (a_1, \ldots, a_n)$ and $\beta = (b_1, \ldots, b_m)$. For any term $a_i$ (resp. $b_i$), there exists a realization of $\crosspair$ such that the vertex corresponding to the term $a_i$ (resp. $b_i$) is adjacent to the vertices corresponding to terms $b_1, \ldots, b_{a_i}$ (resp. $a_1, \ldots, a_{b_i}$).
\end{lemma}
\begin{proof}
	For any $a_i \in \alpha$, the sequence $\pi = (b_1 + m-1, b_2 + m-1, \ldots, b_m + m -1, a_1, \ldots, a_n)$ is a degree sequence associated with $\crosspair$. Since $\crosspair$ corresponds to a split graph having a clique of size $|\beta| = m$, we have $a_1 \leq m$, with equality only if $b_m \ge 1$. Hence $\pi$ is non-increasing, and to finish the proof we apply Theorem~\ref{prop: KW theorem}.
\end{proof}

We now begin our proof of Theorem \ref{prop:Rao_minimal_split}. We first show that the sequence pairs in~\eqref{eq: split cross degree lists} are Rao-minimal. Proceeding as in the proof of Theorem \ref{prop:Rao_minimal_nonsplit}, we define Kleitman-Wang-style reductions and apply them to all larger target pairs, concluding that every target pair other than the ones listed in~\eqref{eq: split cross degree lists} may be successfully reduced to a shorter target pair, so the result follows inductively.

\begin{proof}[Proof of Theorem~\ref{prop:Rao_minimal_split}]
As noted above, each pair in~\eqref{eq: split cross degree lists} is a target pair. To see that the sequence pairs in~\eqref{eq: split cross degree lists} are Rao-minimal, 
we show that sequence pairs $\crosspair$ with length at most 7 correspond to degree sequences of graphs in $HCU$. Let $\alpha = (a_1,\dots,a_n)$ and $\beta=(b_1,\dots,b_m)$, and assume that $n+m \leq 7$. If $m \le 2$ (or $n \le 2$, respectively), then since $\crosspair$ is indecomposable, every term in $\alpha$ (in $\beta$) must be $1$, so the sequence pair satisfies (1a) or (1b) of Theorem~\ref{prop:degree_characterization_split}. If instead $m=3$ and $n=3$, then all terms in $\alpha$ and $\beta$ are $1$ or $2$. Since at least two terms of $\alpha$ must be equal, condition (1a) or (2a) from Theorem~\ref{prop:degree_characterization_split} applies. If $m=3$ and $n=4$, then all terms in $\alpha$ are $1$ or $2$. This yields condition (1a) or (2a) as before unless $\alpha=(2,2,1,1)$; with that choice of $\alpha$, conditions (3a) or (4a) apply unless $b_1 \neq 3$ and $b_m \neq 1$, i.e., $\beta=(2,2,2)$. Finally, the pair $\{(2,2,1,1),(2,2,2)\}$, fits conditions (3b) and (4b). The case $m=4$ and $n=3$ is similar. In every case, Theorem~\ref{prop:degree_characterization_split} shows that $\crosspair$ belongs to a graph in $HCU$, as claimed. 

Hence target pairs with length 8 are necessarily Rao-minimal. It is straightforward to verify that the sequence pairs in~\eqref{eq: split cross degree lists} are precisely these target pairs.

From here, we proceed by induction on the length of a target pair, showing that all target pairs with length at least 9 must Rao-contain a target pair with length that is 1 or 2 less. Let $\crosspair$ be a target pair where $\alpha = (a_1,\dots,a_n)$ and $\beta = (b_1,\dots,b_m)$ and $\alpha,\beta$ satisfy $m+n \ge 9$ with $m,n \ge 3$. (We lose no generality here, for if $m$ or $n$ is at most 2, then Lemma~\ref{prop:EG_nonsplit_useful} or Theorem~\ref{prop:degree_characterization_split} implies that any realization of $\crosspair$ is decomposable or an element of $HCU$.) For each term $a=a_i$ of $\alpha$, define a reduction $R_{a, \alpha}$ to mean removing $a_i$ from $\alpha$ and decrementing each of $b_1,\dots,b_a$ by $1$, reordering terms in the resulting sequence as necessary to keep them in non-increasing order. We denote the reduced sequence pair by $\redpair_{a, \alpha}$. Reductions $R_{b,\beta}$ may be similarly defined for terms $b = b_j$ belonging to $\beta$; the resulting cross-degree pair is denoted $\redpair_{b, \beta}$. We may omit subscripts on reduced sequence pairs $\redpair$ when context makes it clear what is meant. Note that all of these reductions are cross-degree analogs of KW-reductions, and in each the reduced pair $\redpair$ corresponds to a degree sequence resulting from a KW-reduction on a degree sequence corresponding to $\crosspair$.

To help in recognizing the indecomposability of sequence pairs after reductions, we specialize Theorem~\ref{prop:reduction_decomposable} to the context of split graphs.

\begin{corollary} \label{cor: indecomp redpair}
Let $\crosspair$ be an indecomposable sequence pair, and let $\redpair$ be the result of either $R_{a,\alpha}$ or $R_{b,\beta}$ for some term $a$ or $b$. If $\alpha'$ contains no term equal to $0$ or $|\beta'|$, and if $\beta'$ contains no term equal to $0$ or $|\alpha'|$, then $\redpair$ is an indecomposable sequence pair.
\end{corollary}

Given $\crosspair$, if after a reduction $R_{a,\alpha}$ or $R_{b,\beta}$ the pair $\redpair$ satisfies none of the eight conditions in Theorem \ref{prop:degree_characterization_split} and contains no terms in $\alpha'$ equal to $0$ or $|\beta'|$ and no terms in $\beta'$ equal to $0$ or $|\alpha'|$, then Theorem \ref{prop:degree_characterization_split} implies that $\redpair$ is forbidden, Lemma \ref{prop:Rao_induced_split_KW} implies that $\redpair$ is Rao-contained by $\crosspair$, and Corollary~\ref{cor: indecomp redpair} implies that $\redpair$ is indecomposable. In such a case, $\crosspair$ is not Rao-minimal. 

We now consider what happens if every reduction $R_{a, \alpha}$ and $R_{b, \beta}$ yields a pair $\redpair$ that satisfies one of the eight conditions in Theorem \ref{prop:degree_characterization_split}, or contains some term in $\alpha'$ equal to $0$ or $|\beta'|$, or contains some term in $\beta'$ equal to $0$ or $|\alpha'|$. Given $a = a_i$ for some $i$ in $\{1,\dots,n\}$, we claim that the reduction $R_{a, \alpha}$ yields one of these outcomes if and only if one or more of the conditions occurs from the list that follows this paragraph. As in the previous subsection, for most conditions the Roman numeral and accompanying Latin letter indicate which condition from Theorem~\ref{prop:degree_characterization_split} is met by $\redpair$; for instance, condition (II.b.1) below is one case corresponding to condition (2b) in the earlier theorem. Suffixes with Latin numbers are used when a condition from Theorem~\ref{prop:degree_characterization_split} is equivalent to a disjunction of multiple conditions here. Since $\crosspair$ is not the cross-degree sequence pair of a graph in $HCU$, by Theorem~\ref{prop:degree_characterization_split} we see that $a_2 \geq 2$ and $b_2 \geq 2$, and $a_{n-1} \leq m-2$ and $b_{m-1} \leq n-2$. More generally, in arriving at these conditions care has been taken to ensure that none of the conditions from Theorem~\ref{prop:degree_characterization_split} holds for $\crosspair$.

\begin{multicols}{2}
	\begin{enumerate}
		\item [\rm{(I.a)}] $i \le 2$ and $a_3=\dots=a_n=1$. 
		\item [\rm{(I.b.1)}] $a\geq 2$ and $b_2 = 2$ and $b_{a+1} = \dots = b_m =1$. 
		\item [\rm{(I.b.2)}] $b_1 = \dots = b_{a+1}=2$ and $b_{a+2}=\dots=b_m = 1$.
		\item [\rm{(II.a)}] $i \ge n-1$ and $a_1=\dots=a_{n-2} = m-1$. 
		\item [\rm{(II.b.1)}] $b_1=\dots=b_a = n-1$ and $b_{m-1} = n-2$. 
		\item [\rm{(II.b.2)}] $b_1=\dots=b_{a-1} = n-1$ and $b_a=\dots=b_m = n-2$. 
		\item [\rm{(III.a.1)}] $i = 1$ or $a_1=2$; $a_2 = 2$; and $b_1 = \ldots = b_{a + 1} = n-2$. 
		\item [\rm{(III.a.2)}] $i = 1$, $a_1 \ge 3$, $a_2 = 2$, and $b_1 = n-1$.
		\item [\rm{(III.b)}] $i \ge 2$ or $a_2 = m-1$; $a_1 = m-1$; $b_1=3$; and $b_{a+1}\le 2$. 
		\item [\rm{(IV.a.1)}] $i = n$ or $a_n = m-2$; $a_{n-1} = m-2$; and $b_a=\dots=b_m=2$. 
		\item [\rm{(IV.a.2)}] $i = n$, $a_{n-1}=m-2$, $a_n \le m-3$, and $b_m = 1$.
		\item [\rm{(IV.b)}] $i \le n-1$ or $a_{n-1} = 1$; $a_n = 1$; $b_a \ge n-2$; and $b_m = n-3$. 
		\item [\rm{(V)}] $b_a=1$. 
		\item [\rm{(VI)}] $b_{a+1}=n-1$.
	\end{enumerate}
\end{multicols}

\bigskip
As in the non-split argument in the previous subsection, for brevity we illustrate only a representative equivalence from the list above and leave verification of the others to the reader. Specifically, we show the equivalence of condition (2b) from Theorem~\ref{prop:degree_characterization_split}, applied to $\redpair_{a,\alpha}$, and the disjunction of conditions (II.b.1) and (II.b.2). It is straightforward to see that if $\crosspair$ satisfies either (II.b.1) or (II.b.2) above, then the reduced pair $\redpair_{a,\alpha}$ satisfies condition (2b) of Theorem~\ref{prop:degree_characterization_split}. Conversely, if (2b) holds for $\redpair_{a,\alpha}$, then the second-to-last term of $\beta'$ is one less than the length of $\alpha'$, so this term equals $n-2$. Now the second-to-last term of $\beta'$ must equal either $b_{m-1}$ or $b_{m-1}-1$, since even after decrementing and reordering, the value in no position decreases by more than 1 as $\beta$ is changed to $\beta'$. Since $b_{m-1} \leq n-2$, we conclude that both $b_{m-1}$ and the second-to-last term of $\beta'$ equal $n-2$. We must now determine conditions under which the values in these positions of $\beta$ and of $\beta'$ agree. If no reordering of terms involves the $(m-1)$th position, then $a < m-1$ and $b_a > b_{m-1}=n-2$; this requires that $b_1=\dots=b_a=n-1$, and (II.b.1) holds. If reordering of terms involving the $(m-1)$th position does happen after terms of $\beta$ are decremented, then some term of $\beta$ that was decremented originally had value $b_{m-1}$, but any such term must be moved to a position after the $(m-1$)th during reordering; this requires that $b_{a-1}>b_{a} = b_{m-1}=b_m$, and (II.b.2) holds, as desired. 

\bigskip
Given the assumptions outlined above, we assume henceforth the equivalence of the conditions from Theorem~\ref{prop:degree_characterization_split}, as applied to $\redpair_{a,\alpha}$, and conditions (I.a) through (VI) as applied to $\crosspair$. To this point we have not had a reason to distinguish between $\alpha$ and $\beta$, and we have likewise assumed that every reduced pair $\redpair_{b,\beta}$ satisfies a condition from Theorem~\ref{prop:degree_characterization_split} or contains a term in $\alpha'$ or $\beta'$ that would imply that the reduced pair is decomposable. Hence by the arguments above, we may assume that $\crosspair$ also satisfies one of fourteen conditions (I.a$'$), (I.b.1$'$), and so on through (VI$'$), where in each corresponding condition listed above $n$ and $m$ are exchanged, as are the roles played by terms of $\alpha$ and of $\beta$. The assumption here is that a reduction $R_{b,\beta}$ was performed (with $b=b_i$) rather than $R_{a,\alpha}$. For example, condition (III.a.1$'$) is the condition that $i=1$ or $b_1=2$; $b_2=2$; and $a_1=\dots=a_{b+1} = m-2$. 

We now turn to showing that each condition (I.a) through (VI) leads either to a contradiction or to Rao-containment by $\crosspair$ of one of the cross degree sequence pairs listed in~\eqref{eq: split cross degree lists}. At times the conditions (I.a$'$) through (VI$'$) will help our analysis.

For ease of notation throughout, let \[\sum \alpha = \sum_{t=1}^n a_t \text{ and } \sum \beta = \sum_{t=1}^m b_t.\] Since each sum counts the same edges in a split graph, $\sum \alpha = \sum \beta$. When we compute upper and lower bounds on these sums, we notate them first in unsimplified form: (multiplicity of terms equal to largest term)$\cdot$ (value of largest term) $ + $ (multiplicity of terms equal to next lower term) $\cdot $ (value of that term) $+ \ldots$. 

We handle sequence pairs with a``short" $\alpha$ or $\beta$ first. Earlier we showed that $m,n \geq 3$; we begin with the case that $m=3$ and hence $n \geq 6$. By assumption, $\crosspair$ does not satisfy the hypothesis of Theorem~\ref{prop:degree_characterization_split}; it follows that $a_1 = a_2 = 2$ and $a_{n-1} = a_n = 1$, and $3 \leq b_1 \le n-2$ and $2 \leq b_3 \leq n-3$. On the other hand, the reduced pair $\redpair_{a_3,\alpha}$ does satisfy one of the conditions (I.a) to (VI) above. Since we may take $i=3$ and consequently $a \in \{1,2\}$, our known values or bounds for the beginning/ending terms of $\alpha,\beta$ force the applicable condition to be one of the following: (III.a.1) if $a=1$, (III.b), (IV.a.1), or (IV.b). If (III.a.1) holds with $a=1$, then $b_2 = n-2$ and $\sum \beta \geq 2(n-2)+1\cdot 2$ while $\sum \alpha = n+2$, a contradiction since $n \geq 6$. If (III.b) holds, then $b_1=3$ and $b_3 \leq 2$, so $8 \leq n+2 \leq \sum \alpha = \sum \beta \leq 8$, and the resulting equalities imply that $\alpha = (2,2,1,1,1,1)$ and $\beta = (3,3,2)$; in this case $\redpair_{a_3,\alpha} = \{(2,2,1,1,1),(3,2,2)\}$, which appears in \eqref{eq: split cross degree lists}. If instead (IV.a.1) holds, then since $b_1>2$ we must have $a=2$ and $b_2=b_3=2$; but then $\sum \beta \leq n+2$ while $\sum \alpha \geq n+3$, a contradiction. Finally, if (IV.b) holds, then since $b_3=n-3$ we have $3n-9 \leq \sum \beta =\sum \alpha \leq 2n-2$, implying that $n\in\{6,7\}$. Direct verification shows that if $n=6$, the only feasible pairs $\crosspair$ are $\{(2,2,2,1,1,1),(3,3,3)\}$ and $\{(2,2,2,2,1,1),(4,3,3)\}$; in both cases $\redpair_{a_3,\alpha}$ appears in $\eqref{eq: split cross degree lists}$. If $n=7$, the only feasible pair is $\{(2,2,2,2,2,1,1),(4,4,4)\}$. Here $\redpair_{a_3,\alpha} = \{(2,2,2,2,1,1),(4,3,3)\}$, one of the sequences just considered; we see again that $\crosspair$ Rao-contains an element of $\eqref{eq: split cross degree lists}$.

To this point we have not distinguished between $\alpha$ and $\beta$ in the discussion, so we may exchange the names of $\alpha,\beta$ and apply the preceding arguments if $n=3$. Breaking the symmetry between $\alpha$ and $\beta$ somewhat, we assume henceforth that $4 \leq m \leq n$. Since $m+ n \geq 9$, we have $n \geq 5$. 

To make the remaining proof more efficient, we consider the conditions from our list in the following order: (III.a.1), (IV.a.1), (III.a.2), (IV.a.2), (I.b.1), (II.b.1), (I.b.2), (II.b.2), (I.a), (II.a), (III.b), (IV.b), and lastly (V) and (VI) together. 

Suppose first that $\crosspair$ satisfies (III.a.1). Here $i = 1$ or $a_1 = 2$; $a_2 = 2$; and $b_1 = \ldots = b_{a + 1} = n-2$. (Recall that $a=a_i \geq 1$.) Since the terms of $\alpha$ and $\beta$ are non-increasing, \begin{equation}\label{eq: III.a.1} (a + 1)(n-2) + (m - a - 1)\cdot 1 \leq \sum\beta = \sum \alpha \le (n-1)\cdot 2 + 1\cdot a.\end{equation}
Rearranging this inequality yields $(n-4)(a-1) \leq 5 - m$. Since $m \geq 4$ and $n \geq 5$ and $a \geq 1$, both sides of this latter inequality are non-negative; hence $m \in \{4,5\}$, and we see that that $a=1$ unless $a=2$ and $n=5$ and $m=4$. In the latter case, the resulting equality forces $\crosspair = \{(2,2,2,2,2),(3,3,3,1)\}$, which satisfies (4a) of Theorem~\ref{prop:degree_characterization_split}, contradicting the assumption that $\crosspair$ was forbidden. If instead $a=1$, then we may write $\alpha = (2,2,a_3,\dots,a_{n-1},1)$ and $\beta = (n-2,n-2,b_3,\dots,b_m)$. Consider the reduction $R_{a_2}$; we claim that $\redpair$ is a target sequence pair with smaller length. Indeed, this pair is indecomposable by Corollary~\ref{cor: indecomp redpair}. To see that it is forbidden as well, we consult Theorem~\ref{prop:degree_characterization_split} as applied to $\alpha' = (2,a_3,\dots,a_{n-1},1)$ and $\beta' = (n-3,n-3,b_3,\dots,b_m)$, though the terms of $\beta'$ may be reordered so as to be non-increasing. Observe that $|\alpha'|=n-1$, and conditions (1b), (2a), (3b), and (4a) clearly cannot hold due to our lower bounds on $m$ and $n$. If (1a) were to hold for $\redpair$, then $\alpha=(2,2,1,\dots,1)$ and the inequality \eqref{eq: III.a.1} yields $2(n-2)+1(m-2) \leq n+2$ and hence $m+n \leq 8$, a contradiction. The condition (2b) cannot hold for $\redpair$, as $\beta'$ has at least two terms equal to $n-3$. If (3a) were to hold for $\redpair$, then $b_3=n-2$, and the inequality \eqref{eq: III.a.1} yields $3(n-2)+1(m-3) \leq 2(n-1)+1$, again leading to the contradiction $m+n \leq 8$. Finally, if $\redpair$ satisfied (4b) from Theorem~\ref{prop:degree_characterization_split}, then $b_m \geq n-3$, so the inequality \eqref{eq: III.a.1} implies that $(n-2)2+(n-3)(m-2) \leq 2(n-1)+1$, leading to $(n-3)(m-2)\leq 3$; this is a contradiction as $n \geq 5$ and $m \geq 4$. Hence $\redpair$ is a target sequence pair, and $\crosspair$ is not a Rao-minimal sequence pair.

Suppose instead that $\crosspair$ satisfies condition (IV.a.1). The complement of a split graph realization of $\crosspair$ has degree sequence pair $\{\overline{\alpha},\overline{\beta}\}$, where $\overline{\alpha}=(m-a_n,\dots,m-a_1)$ and $\overline{\beta} = (n-b_m,\dots,n-b_1)$. Since $\crosspair$ satisfies (IV.a.1), it is straightforward to verify that $\{\overline{\alpha},\overline{\beta}\}$ satisfies (III.a.1). By the previous argument and Remark \ref{prop:Rao_min_indecomp_complement}, we see that $\crosspair$ is not Rao-minimal.

Suppose next that $\crosspair$ satisfies (III.a.2). Here $a = a_1 \ge 3$, $a_2 = 2$, and $b_1 = n-1$. We will see that for any such $\crosspair$, the reduction $R_{a_{n-1},\alpha}$ yields a target pair $\redpair$. First, prior assumptions ensure that $\redpair_{a_{n-1},\alpha}$ has no terms in $\beta'$ equal to $0$, nor terms in $\alpha'$ equal to $0$ or $|\beta'|$. If $\redpair$ contains a term in $\beta'$ equal to $|\alpha'| = n-1$, then $b_{k+1} = n-1$, where $k = a_{n-1}$. We find $\sum \beta \ge (k+1)(n-1) + (m - k - 1)\cdot 1$ and $\sum \alpha \le 1\cdot(m-1) + (n-3)\cdot 2 + 2\cdot k$, and $\sum \beta = \sum \alpha$ leads to $(n-4)(k - 1) \le -1$, a contradiction. It follows from Corollary~\ref{cor: indecomp redpair} that $\redpair$ is indecomposable. Contrary to our claim that $\redpair$ is a target pair, suppose that one of the other conditions from Theorem~\ref{prop:degree_characterization_split} holds for $\redpair$. Our bounds $4 \leq m \leq n$ and the present case's starting conditions on $a_1$, $a_2$, and $b_1$ rule out (1a), (2a), (3a), and (3b). If (4a) or (4b) were to occur for $\redpair$, then $\crosspair$ would also satisfy (4a) or (4b), respectively, a contradiction. If $\redpair$ satisfies (1b), then we must have $a_{n-1} = 2$ and $b_2 = 2$. Thus $\beta = (n-1, 2, 1, \ldots, 1)$ and $\alpha = (a_1, 2, \ldots, 2, a_n)$, which yields a contradiction to $\sum \alpha = \sum \beta$. If $\redpair$ satisfies (2b), then $\beta' = (n-2, \ldots, n-2, b_m')$ for some $b_m'$, so $\sum \beta \ge a_{n-1}(n-1) + (m-1-a_{n-1})(n-2) + 1\cdot 1$. However, $\sum \alpha \le 1\cdot(m-1) + (n-3)\cdot2 + 2\cdot a_{n-1}$, and the relation $\sum \beta = \sum \alpha$ yields $(m-3)(n-3) \le a_{n-1}-1$, a contradiction. Hence $\redpair$ is a target sequence, as claimed, and $\crosspair$ is not Rao-minimal.

If $\crosspair$ satisfies condition (IV.a.2), then with an argument similar to that of condition (IV.a.1) above, we can apply condition (III.a.2) to the complementary pair $\{\overline{\alpha},\overline{\beta}\}$ to conclude that $\crosspair$ is not Rao-minimal.

At this point we comment on our process. As we have considered the conditions (III.a.1) and (III.a.2) above and deduced results about conditions (IV.a.1) and (IV.a.2), we have proceeded by taking an arbitrary target sequence pair $\crosspair$ satisfying the given condition and examining a reduced pair $\redpair$. After showing that that $\redpair$ must be indecomposable, we show that $\redpair$ is also forbidden (and hence a target sequence pair) by checking each of the conditions (1a) through (4b) from Theorem~\ref{prop:degree_characterization_split}. As we continue the proof in the coming paragraphs, moving through other conditions from the list (I.a) through (VI), from this point on it will not be necessary to check conditions (3a) or (4a) from Theorem~\ref{prop:degree_characterization_split} while verifying that $\redpair$ is forbidden, as long as $\redpair=\redpair_{a,\alpha}$ for some term $a \in \alpha$. This is because of the respective equivalences already mentioned between (3a) and (4a), applied to $\redpair$, and the disjunctions ``(III.a.1) or (III.a.2)'' and ``(IV.a.1) or (IV.a.2)'' applicable to $\crosspair$, as well as the arguments thus far in the proof.

Resuming our progress, suppose now that $\crosspair$ satisfies (I.b.1) for some term $a$ of $\alpha$, so $b_2=2$ and $b_{a+1}=1$. Here $a > 1$ and $\sum \beta \le 1 \cdot (n-1) + (a-1)\cdot 2 + (m - a) \cdot 1= n+m+a-3$. If $a_n \ge 2$, then $\sum \alpha \ge 1\cdot a + (n-1)\cdot 2$ and we have $2n + a - 2 \le n + m + a -3 $, which is a contradiction since $m \leq n$. Hence $a_n = 1$. We now show that the reduction $R_{a_n,\alpha}$ yields a target sequence. First, it is clear that $\redpair_{a_n,\alpha}$ has no terms in $\alpha'$ equal to $0$ or $|\beta'|$ (since $\alpha$ has none) and no terms in $\beta'$ equal to $0$ or $|\alpha'|$, since $b_1 \ge 2$ and $b_2 \le n-2$. By Corollary~\ref{cor: indecomp redpair}, $\redpair$ is indecomposable. Consulting conditions in Theorem~\ref{prop:degree_characterization_split} as they might apply to $\redpair$,
we easily rule out (1a) as it implies that $\crosspair$ satisfies (1a) as well; (2b) and (4b) likewise cannot hold given our lower bound on $n$. Condition (1b) implies that $b_1 = 2$ and $b_3 = 1$, so $\sum \beta = m+2$. In this case, since (1a) does not hold, $a_2 \geq 2$, so $m+2=\sum \alpha \geq n+2$. Since $n \geq m$ we conclude that $\crosspair = \{(2,2,1,\ldots, 1), (2,2,1,\ldots,1)\}$ with $m = n \geq 5$; this pair Rao-contains the pair $\{(2,2,1,1),(2,2,1,1)\}$ from \eqref{eq: split cross degree lists}, contradicting our assumption that $\crosspair$ is Rao-minimal. Thus (1b) cannot hold. If instead $\redpair$ satisfies (2a), then $a_{n-2} = m-1$, so $\sum \alpha \ge (n-2)(m-1)+2\cdot 1$, and $\sum \beta \le 1\cdot (n-1) + (m-2)\cdot 2 + 1\cdot 1$. Combining these inequalities via $\sum \alpha = \sum \beta$ yields $(n-4)(m-2) \leq 0$, a contradiction, so (2a) does not hold for $\crosspair$. Lastly, if $\redpair$ satisfies (3b), then $b_1 \le 3$; we proceed through the possibilities for $b_1$. If $b_1 \leq 2$, then since neither (1a) nor (1b) holds, we have $a_2 \geq 2$ and $b_1=b_2=b_3=2$, so $a \geq 3$ and the equality $\sum \alpha = \sum \beta$ implies that $n=m$, that $\alpha = (a,2,1,\dots,1)$, and that $\beta$ has $a$ terms equal to 2 and $n-a$ terms equal to $1$. Here $\crosspair$ Rao-contains $\{(2,2,1,1),(2,2,1,1)\}$, a contradiction. Hence we take $b_1=3$; here $1\cdot(m-1) + 1\cdot2 + (n-2)\cdot 1 \leq \sum \alpha = \sum \beta \leq 1\cdot3 + (m-2)\cdot 2 + 1\cdot 1$, yielding $n \le m+1$. If $n = m+1$, then the equality forces $\alpha = (m-1, 2, 1, \ldots, 1)$ and $\beta = (3,2, \ldots, 2, 1)$, while if $n = m$, then it is possible to have $\beta$ as above and $\alpha = (m-1, 3, 1, \ldots, 1)$ or $\alpha = (m-1, 2, 2, 1, \ldots, 1)$, or to have $\beta = (3, 2, \ldots, 2, 1, 1)$ and $\alpha = (m-1, 2, 1,\ldots, 1)$. In these four cases, $\crosspair$ respectively Rao-contains $\{(2,2,1,1,1),(3,2,2)\}$, $\{(3,3,1,1),(3,2,2,1)\}$, $\{(3,2,2,1),(3,2,2,1)\}$, and $\{(3,2,1,1),(3,2,1,1)\}$, which appear in \eqref{eq: split cross degree lists}. (To see this, consider the cross-degree sequence pairs formed when the high-degree term and $m-4$ terms equal to 1 in $\alpha$ are successively decremented by $m-4$ terms equal to 2 from $\beta$.) Hence no target cross-degree pairs satisfying (I.b.1) are Rao-minimal.

If $\crosspair$ satisfies condition (II.b.1), then we can apply condition (I.b.1) to the complementary pair $\{\overline{\alpha}, \overline{\beta}\}$ to see that $\crosspair$ is not Rao-minimal.

Suppose next that $\crosspair$ satisfies (I.b.2) for some term $a$ in $\alpha$, so $\beta$ has $a+1$ terms equal to 2 and $m-a-1$ terms equal to 1. Since $\crosspair$ is forbidden, it cannot satisfy condition (3b) from Theorem~\ref{prop:degree_characterization_split}, so $a_1 \leq m-2$. Hence $a \leq m-2$ as well, and we compute $\sum \beta = m+a+1 \leq 2m-1$. Let $\redpair$ be the sequence pair resulting from the reduction $R_{b_m,\beta}$; we will show that $\redpair$ is a target sequence pair. Notice that $\redpair$ has no terms in $\beta'$ equal to $0$ or $|\alpha'|$ (since $\beta$ does not), and it has no term in $\alpha'$ equal to $0$ or $|\beta'|$, since $2 \leq a_2 \leq a_1 \leq m-2$. By Corollary~\ref{cor: indecomp redpair}, $\redpair$ is indecomposable. We now check $\redpair$ against the conditions in Theorem~\ref{prop:degree_characterization_split}. Note that we must check (3a) and (4a) here since the Kleitman-Wang reduction is indexed in $\beta$ and not $\alpha$. Clearly (1b), (2b), (3a), and (4b) cannot hold because of $\beta$'s structure and our bounds on $m$ and $n$. If $\redpair$ satisfies (1a), then since $a_2 \geq 2$ we have $a_1 = a_2 = 2$ and $a_3 = 1$. This yields $\alpha = (2,2,1, \ldots, 1)$ and either $\beta = (2,2,2,1, \ldots, 1)$ or $\beta = (2,2,1, \ldots, 1)$; the length of $\beta$ is determined by the relation $\sum\alpha = \sum \beta$. In either case, it is straightforward to see that the sequence $\{(2,2,1,1),(2,2,1,1)\}$ from \eqref{eq: split cross degree lists} is Rao-contained by $\crosspair$. If $\redpair$ instead satisfies (2a), then $a_{n-1} \geq a'_{n-1} = m-2$, so we find $\sum \alpha \ge (n-1)(m-2) +1\cdot 1>2m-1\geq\sum\beta$, a contradiction. If $\redpair$ satisfies (3b), then $a_1 = a_2 = m-2$, so $2m-1 \geq \sum \beta = \sum \alpha \ge 2(m-2) + (n-2)\cdot 1=2m+n-6$. Hence $n = 5$, and equality, together with the condition that $4 \leq m \leq n$, implies that $\crosspair$ is either $\{(2,2,1,1,1),(2,2,2,1)\}$ or $\{(3,3,1,1,1),(2,2,2,2,1)\}$, both of which Rao-contain $\{(2,2,1,1),(2,2,1,1)\}$ from \eqref{eq: split cross degree lists}. Finally, if $\redpair$ satisfies (4a), then $a_n \geq m-3$ and $b_{m-1} = 1$, so $n(m-3) \le \sum \alpha = \sum \beta \leq 2m-2$; we obtain $(m-3)(n-2)\leq 4$, which holds only when holds when $m=4$ and $n \le 6$. However, the bounds on $\sum \alpha$ and $\sum \beta$ also force $\sum \alpha = 6$, which is impossible as $a_2 \geq 2$.

If $\crosspair$ satisfies condition (II.b.2), then the complementary pair $\{\overline{\alpha}, \overline{\beta}\}$ satisfies (I.b.2), showing that $\crosspair$ is not Rao-minimal. Note also that the disjunctions ``(I.b.1) or (I.b.2)'' and ``(II.b.1) or (II.b.2),'' applicable to $\crosspair$, have been seen to be equivalent to Theorem~\ref{prop:degree_characterization_split}'s conditions (1b) and (2b) on a reduced pair $\redpair_{a,\alpha}$. We may then assume henceforth that whenever we check a reduced pair $\redpair_{a,\alpha}$ to see if its realization is in HCU (as it should be for $\crosspair$ to be Rao-minimal), it suffices to check conditions other than (1b) or (2b) (or (3a) or (4a), as established earlier).

Suppose now that $\crosspair$ satisfies (I.a), so $a_3 = \ldots = a_n = 1$. We will handle this this possibility with two choices for $\redpair$, depending on the size of $b_1$. Suppose first that $b_1 \geq 4$ or that $b_1=b_2=3$. In this case let $\redpair = \redpair_{a_n,\alpha}$. Clearly no term of $\alpha'$ equals $0$ or $|\beta'|$, and no term of $\beta'$ equals $0$; more care is needed to show that no term of $\beta'$ equals $|\alpha'| = n-1$. If instead this were to happen, then $b_1=b_2 = n-1$, and $2(n-1)+(m-2)\cdot 1 \leq \sum \beta = \sum \alpha \leq 2(m-1)+(n-2)\cdot 1$, which implies that $n \leq m$ and hence $m=n\geq 5$. Equality in the inequalities above requires that $\alpha = \beta = (n-1, n-1, 1, \ldots, 1)$, and $\crosspair$ now fails the condition of Lemma \ref{prop:graphic_Gale_Ryser} at $k = 2$, a contradiction. Thus $\beta'$ contains no terms equal to $|\alpha'|$, and $b_2 \leq n-2$. Turning now to Theorem~\ref{prop:degree_characterization_split} as it applies to $\redpair$, we verify that $\redpair$ is forbidden by checking conditions (1a), (2a), (3b), and (4b). Indeed, (1a) cannot hold, since it does not hold for $\crosspair$, and (2a) fails since $m \geq 4$. Since $b_1\geq 4$ or $b_1=b_2\geq 3$, we have $b'_1 > 2$ and hence (3b) cannot hold. If $\redpair$ satisfies (4b), then $b_m \geq n-3$ and $b_1 \geq n-2$, so $1\cdot (n-2)+(m-1)(n-3) \leq \sum \beta = \sum \alpha \leq 2(m-1)+(n-2)\cdot 1$, which yields $(m-1)(n-5) \leq 0$. Since $m \geq 4$ and $n \geq 5$, we then have $n=5$ and equality in the degree sum inequalities forces $\alpha = (m-1, m-1, 1, 1, 1)$ and $\beta = (3,2,\dots,2)$. This contradicts our hypothesis that $b_1 \geq 4$ or $b_1=b_2=3$; we conclude that $\redpair$ is forbidden, so $\crosspair$ is not Rao-minimal. 

Continuing with case (I.a), now suppose that $b_1 \leq 3$ and $b_2 \leq 2$, and let $\redpair = \redpair_{b_m,\beta}$. No term of $\beta'$ equals $0$ or $|\alpha'|$, and since $a_2 \geq 2$, no term of $\alpha'$ equals 0. To see that no term of $\alpha'$ is equal to $|\beta'|$, note that otherwise would require that $a_1=a_2=m-1$ and $b_m=1$, yielding $2(m-1)+(n-2) \cdot 1 = \sum \alpha = \sum \beta \leq 1 \cdot 3+(m-2)\cdot 2+1\cdot 1$, a contradiction. Considering the conditions from Theorem~\ref{prop:degree_characterization_split} as applied to $\redpair$, we see that (1b) cannot hold for $\redpair$ since it does not hold for $\crosspair$, and the sizes of $m$ and $n$ quickly preclude (2a), (2b), (3a), (4a), or (4b) from happening. It remains to check the cases when $\redpair$ satisfies (1a) or (3b). If (1a) holds for $\redpair$, then either $b_m=1$ and $a_1=a_2=2$, or $b_m=2$ and $a_2 = 2$. In the former case, where $\alpha=(2,2,1,\dots,1)$, one can verify that $\crosspair$ Rao-contains $\{(2,2,1,1),(2,2,1,1)\}$. In the latter case we have $\alpha = (a_1, 2, 1, \ldots, 1)$ and $\beta = (b_1, 2, \ldots, 2)$. Since $\crosspair$ does not satisfy condition (3b) from Theorem~\ref{prop:degree_characterization_split}, if $b_1=2$ then $a_1\leq m-2$, and $\crosspair = \{(a_1,2,1,...,1),(2,...,2)\}$, which again Rao-contains $\{(2,2,1,1),(2,2,1,1)\}$; this can be seen by induction on $m$, where the induction step involves taking a realization where a vertex with cross-degree $b_m = 2$ is adjacent to vertices with cross-degrees $a_1,1$ (if $a_1>2$) or $1,1$. Assume now that $b_1 = 3$, so $\alpha = (a_1, 2, 1, \ldots, 1)$ and $\beta = (3, 2, \ldots, 2)$; this cross-degree sequence pair Rao-contains $\{(2,2,1,1,1),(3,2,2)\}$, since $\sum \alpha = \sum \beta$ forces $m+2 \leq n \leq 2m-1$ and we may use $m-3$ terms of $2$ from $\beta$ to iteratively decrement from $\alpha$ both a term equal to 1 and the term of highest degree. This shows that $\crosspair$ is not Rao-minimal in this case. If $\redpair$ instead satisfies (3b), then $2 \geq b'_1 = b_1 \geq 2$; since $\crosspair$ itself does not satisfy (3b), we must have $a_1 < m-1$, but since $a'_1 = m-2$, we find that $a_1=a_2=m-2$ and $b_m=1$. Then $2m+n-6 = \sum \alpha = \sum \beta \leq (m-1) \cdot 2 +1\cdot 1$, and we obtain $n=5$. The pair $\crosspair$ is then $\{(3,3,1,1,1), (2,2,2,2,1)\}$, which Rao-contains $\{(2,2,1,1),(2,2,1,1)\}$, so $\crosspair$ is not Rao-minimal.

Proceeding as in previous cases, we note that if $\crosspair$ satisfies condition (II.a), then $\{\overline{\alpha},\overline{\beta})$ satisfies the case (I.a) just considered and hence $\crosspair$ is not Rao-minimal. Furthermore, the respective conditions (I.a) and (II.a), applicable to $\crosspair$, are equivalent to the conditions (1a) and (2a) from Theorem~\ref{prop:degree_characterization_split} on a reduced pair $\redpair_{a,\alpha}$, so we henceforth need not check conditions (1a) or (2a) (or the previously considered conditions (1b), (2b), (3a), (4a)) when considering whether a reduced pair $\redpair_{a,\alpha}$ belongs to HCU.

Suppose that $\crosspair$ satisfies (III.b) for some term $a=a_i$ of $\alpha$, so $b_1 = 3$, $b_{a+1} \leq 2$, and $a_1 = m-1$, while $i \geq 2$ or $a_2 = m-1$. We may assume that $a_3 \ge 2$ since we have considered the case (I.a) above. We consider first the subcase in which $b_1=\cdots=b_{a_n+1}=3$. Let $\redpair=\redpair_{a_n,\alpha}$. Clearly no term of $\alpha'$ is equal to $0$ or $|\beta'|$, and the values $b_1=b_{a_{n}}=3$ ensure that no term of $\beta'$ is equal to $|\alpha|$ or $0$. By Corollary~\ref{cor: indecomp redpair}, $\redpair$ is indecomposable. Checking $\redpair$ for the conditions of Theorem~\ref{prop:degree_characterization_split} not previously handled, we see that $b'_1 > 2$, so (3b) does not hold. If $\redpair$ satisfies (4b), then $a_{n-1} = 1$ and $b_m \geq n-3$; since (III.b) requires the last term of $\beta$ to be at most 2, this forces $n=5$ and $b_m = 2$. Likewise, $a_n =1$, so $b_2=b_{a_n+1} = 3$. Listing the possibilities shows that there are six cross-degree sequence pairs $\crosspair$ with $\alpha = (m-1,a_2,a_3,1,1)$ and $\beta = (3,3,b_3,b_4)$ or $\beta = (3,3,b_3,b_4,b_5)$, where $a_2 \geq 2$ and the last term of $\beta$ is no more than 2, and $\sum \alpha = \sum \beta$. These pairs are $\{(3, 3, 2, 1, 1), (3, 3, 2, 2)\}$, 
$\{(3, 3, 3, 1, 1), (3, 3, 3, 2)\}$, 
$\{(4, 3, 3, 1, 1), (3, 3, 2, 2, 2)\}$,
$\{(4, 4, 2, 1, 1), (3, 3, 2, 2, 2)\}$, 
$\{(4, 4, 3, 1, 1), (3, 3, 3, 2, 2)\}$, and 
$\{(4, 4, 4, 1, 1), (3, 3, 3, 3, 2)\}$. Observe that each of these Rao-contains $\{(3,2,1,1)(3,2,1,1)\}$ or $\{(3,3,1,1)(3,2,2,1)\}$, so in this subcase $\crosspair$ is not Rao-minimal.

Still assuming that $\crosspair$ satisfies (III.b), consider the subcase in which $b_{a_n+1}\leq 2$. Here $\sum \beta < 3m \leq 3n$, so $a_n \leq 2$ and hence $b_3 \leq 2$. Let $\redpair=\redpair_{b,\beta}$, where $b=b_m$. Clearly $\beta'$ contains no terms equal to $0$ or $|\alpha'|$, and since $b_m \leq 2$ and $a_2 \geq 2$, clearly $\alpha'$ contains no 0 term. If $a'_1=m-1=|\beta'|$, then $a_1=a_2=m-1$; we continue to assume $a_3 \ge 2$. This leads to $\sum \alpha \ge 2(m-1) + 1 \cdot 2 + (n-3)\cdot 1 = 2m + n - 3$ and $\sum \beta \le 2 \cdot 3 + (m-2) \cdot 2 = 2m + 2$; we conclude that $n=5$ and $\crosspair = \{(m-1,m-1,2,1,1),(3,3,2,\dots,2)\}$, but then $a'_1 = m-2$. Hence we may assume by Corollary~\ref{cor: indecomp redpair} that $\redpair$ is indecomposable and check $\redpair$ against the conditions in Theorem~\ref{prop:degree_characterization_split}. Conditions (1a) and (3b) fail to hold in light of our known values of terms in $\crosspair$, and because $n \geq 5$, conditions (2b), (3a) and (4b) similarly fail. Since $\crosspair$ does not satisfy (1b), neither can $\redpair$. If $\redpair$ satisfies (4a), then knowing that $b_m \leq b_{m-1}=1$ and $m-3 \leq a_n \leq 2$ forces $m \leq 5$. Since $a_1=m-1$ and $a_3 \geq 2$, while $b_1 = 3$ and $b_3 \leq 2$ and $b_{m-1}\leq 1$, we find \[m+n = 1 \cdot (m-1) + 2\cdot 2 + (n-3)\cdot 1 \leq \sum \alpha =\sum \beta \leq 2\cdot 3 + (m-4) \cdot 2 + 2\cdot 1 = 2m.\] Since $n \geq m$ we see that $m=n=5$ and equality must hold in each of the inequalities, forcing $a_n = 1$ and $b_{2}=3$, which contradicts the assumption that $b_{a_n+1} \leq 2$. Hence $\redpair$ cannot satisfy (4a). If instead $\redpair$ satisfies (2a), then $a_{n-1} \geq m-2$; then $2m+2 \geq \sum \beta = \sum \alpha \geq 1 \cdot (m-1)+(n-2)(m-2)+1 \cdot 1$. This would imply that $(m-2)(m-3) \leq 4$, which forces $m=2$ and $n=3$ and $\crosspair = \{(3,2,2,2,1),(3,3,2,2)\}$, contradicting $b_{a_n + 1} \le 2$. 

If $\crosspair$ satisfies condition (IV.b), then the complementary pair $\{\overline{\alpha}, \overline{\beta}\}$ satisfies (III.b), showing again that $\crosspair$ is not Rao-minimal.

If none of conditions (I.a) through (IV.b) hold for any reduction $R_{a_i,\alpha}$, then for each term $a$ of $\alpha$, the reduction $R_{a,\alpha}$ must satisfy condition (V) or (VI). However, this leads to a contradiction. First, if the reduced pairs $\redpair_{a_i,\alpha}$ satisfy (V) for all terms $a_i$ in $\alpha$, then $b_{a_n} = \dots = b_m = 1$; since $b_2 \ge 2$, we have $a_n \ge 3$. Hence $\sum_{i=1}^{a_n - 1} b_i = \sum \beta - (m - a_n + 1)$ and $\sum_{i=1}^n \min\{a_i, a_n - 1\} = n(a_n - 1)$. By Lemma \ref{prop:graphic_Gale_Ryser}, the $(a_n - 1)$th Gale--Ryser inequality implies that $\sum \beta \le m-a_n+1 +na_n - n \le na_n - a_n + 1$. Since $\sum \beta = \sum \alpha \ge na_n$ and $a_n \ge 3$, this is a contradiction. 

Suppose instead that condition (V) fails for some term $a_i$ in $\alpha$. Let $t$ be the least subscript on a term $a_t$ of $\alpha$ such that $b_{a_t} > 1$. By previous assumptions, condition (VI) then holds for $\redpair_{a_{t}+1,\alpha}$. Thus $b_{a_t + 1} = n-1$, $b_{a_{t-1}} = 1$, and $a_{t-1} \ge a_t + 2$. By minimality of $t$, $a_{t-1} > a_t$. Via the $(a_t + 1$)th Gale--Ryser inequality, Lemma \ref{prop:graphic_Gale_Ryser} implies that $$\sum_{i=1}^{a_t + 1} b_i = (n-1)(a_t + 1) \le \sum_{i=1}^n \min\{a_t + 1, a_i \} = (t-1)(a_t + 1) + \sum_{i=t}^n a_i.$$ We manipulate this inequality to show that $\sum_{i=t}^n a_i \geq (n-t)(a_t + 1)$. Moreover, the monotonicity of $\alpha$ implies $(n-t+1)a_t \geq \sum_{i=t}^n a_i$. It follows that $a_t \geq n-t$, that is, $n \le t + a_t$. However, the $(a_{t-1}-1)$th Gale--Ryser inequality implies $$\sum \beta - (m - a_{t-1} + 1) = \sum _{i=1}^{a_{t-1}-1} b_i \le \sum_{i=1}^n \min\{a_{t-1}-1, a_i \} = (t-1)(a_{t-1}-1) + \sum_{i=t}^n a_i = \sum \alpha - \sum_{i=1}^{t-1} (a_i - a_{t-1} + 1).$$ This is equivalent to $ \sum _{i=1}^{t-1} (a_i - a_{t-1}) \le m - a_{t-1} - t + 2$. Since the left side is non-negative, we have $a_{t-1} + t \le m + 2$, or equivalently, $a_t + t \le m \le n$. 

We have shown the $(a_t + 1$)th Gale--Ryser inequality implies $n \le t + a_t$ and the $(a_{t-1}-1)$th Gale--Ryser inequality implies $n \ge t + a_t$. Hence $n=t+a_t$, both Gale--Ryser inequalities hold with equality, and so do all other bounds used in their computation. Specifically, $a_{t-1} = a_t + 2$, so $n-1 = b_{a_t+1} > b_{a_t+2} = b_{a_{t-1}} = 1$. Thus Lemma \ref{prop:cross_degree_GaleRyser}, applied at index $a_t+1$ of $\beta$, implies that $\crosspair$ is decomposable, for a contradiction. This completes the proof.
\end{proof}

\section{A forbidden induced subgraph characterization of $HCU$} \label{sec: forbidden subgraphs}

We use Barrus and Hartke's \cite{BaHa15} theory of degree-sequence-forcing sets to establish a forbidden induced subgraph characterization of $HCU$. We first define a class of graphs to be \textit{characterized by its degree sequences} if for every degree sequence $d$, if one realization belongs to the class, then every realization belongs to the class. A set $\mathcal{F}$ of graphs is \emph{degree-sequence-forcing} (\emph{DSF}) if the set of $\mathcal{F}$-free graphs (in the induced subgraph context) is characterized by its degree sequences. In other words, a graph $G$ is $\mathcal{F}$-free if and only if every other realization of the degree sequence of $G$ is $\mathcal{F}$-free. If a set $\mathcal{F}$ is not $DSF$, then there exists a degree sequence $d$ with realizations $G$ and $G'$ such that $G$ is $\mathcal{F}$-free and $G'$ is not. As in \cite{BaHa15}, we call such a set an \emph{$\mathcal{F}$-breaking pair}. When $\mathcal{F}$ is finite, it is possible to identify whether it is DSF without needing to check all degree sequences. We present this result in the following lemma. 

\begin{lemma} [\cite{BaHa15}, Lemma 2.1]
\label{prop:DSF_Barrus_Hartke}
    Let $\mathcal{F}$ be a set of graphs, each with at most $n$ vertices. If the subset of $\mathcal{F}$-free graphs of order up to $n + 2$ is characterized by its degree sequences, then the whole set of $\mathcal{F}$-free graphs is characterized by its degree sequences, that is, $\mathcal{F}$ is DSF.
\end{lemma}

The proof of Lemma \ref{prop:DSF_Barrus_Hartke} relies on a graph operation called a \textit{2-switch}: given a $4$-vertex subset of a graph with edges $ab, cd$ and non-edges $ad, bc$, a 2-switch deletes $ab, cd$ and adds $ad, bc$. Critically, given two graphs $G, G'$, there exists a finite sequence of 2-switches transforming $G$ to $G'$ if and only if $G$ and $G'$ have the same degree sequence \cite{FuHoMc65}. 

Since $HCU$ is hereditary by definition, there exists a set $\mathcal{F}_{HCU}$ of minimal forbidden induced subgraphs, such that $G \in HCU$ if and only if $G$ is $\mathcal{F}_{HCU}$-free. Since Theorems \ref{prop:degree_characterization_nonsplit} and \ref{prop:degree_characterization_split} give a degree sequence characterization of $HCU$, it follows that $\mathcal{F}_{HCU}$ is a DSF set. By Remark \ref{prop:Rao_min_indecomp_complement}, every realization of a Rao-minimal degree sequence is in $\mathcal{F}_{HCU}$. Let $\mathcal{Y}$ be the set of graphs shown in Figure \ref{fig:big_graphs} together with their complements and the realizations of each Rao-minimal sequence identified in Theorems \ref{prop:Rao_minimal_nonsplit} and \ref{prop:Rao_minimal_split}. Using Lemma \ref{prop:DSF_Barrus_Hartke}, we give a forbidden induced subgraph characterization of $HCU$.

\begin{figure}
\centering
  \includegraphics[height=4cm]{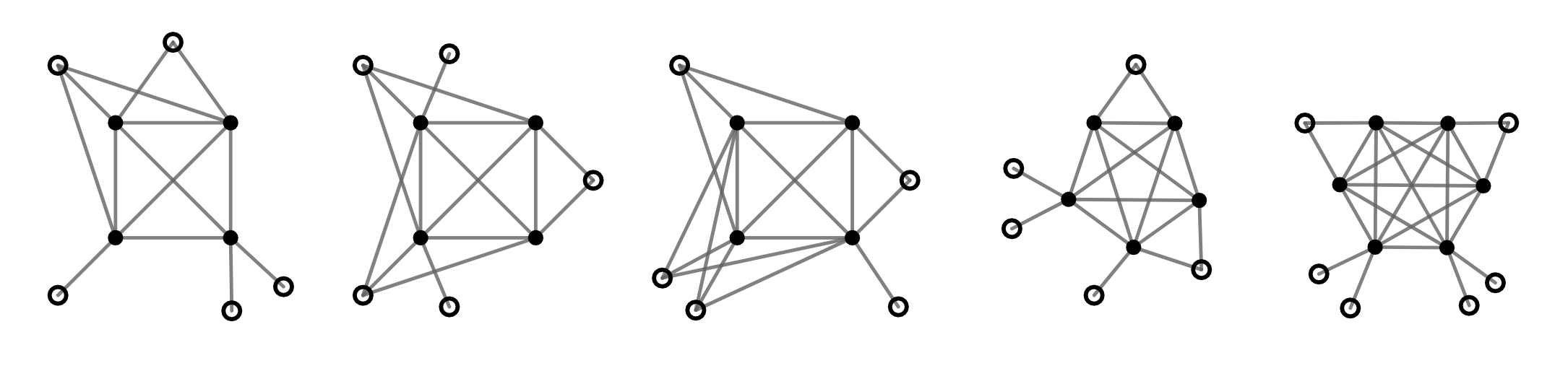}
  \caption{The five split graphs which, together with their complements, are the subset of $\mathcal{F}_{HCU}$ whose degree sequences are not Rao-minimal.}
  \label{fig:big_graphs}
\end{figure}

\begin{theorem}
    The minimal forbidden induced subgraphs of $HCU$ are the graphs shown in Figure \ref{fig:big_graphs}, their complements, and the graphs realizing the Rao-minimal forbidden degree sequences identified in Theorems \ref{prop:Rao_minimal_nonsplit} and \ref{prop:Rao_minimal_split}. That is, $\mathcal{F}_{HCU} = \mathcal{Y}$.
\end{theorem}
\begin{proof}
First, we show $\mathcal{Y}$ is contained in $\mathcal{F}_{HCU}$. Using Lemmas~\ref{prop:EG_nonsplit_useful} and~\ref{prop:cross_degree_GaleRyser} and Theorems~\ref{prop:degree_characterization_nonsplit} and~\ref{prop:degree_characterization_split}, it is straightforward to verify from the degree sequences that each graph in $\mathcal{Y}$ is indecomposable and not in $HCU$. We can also see that any proper induced subgraph of these graphs is in $HCU$, implying that $\mathcal{Y} \subseteq \mathcal{F}_{HCU}$. 

We now show $\mathcal{F}_{HCU}$ is contained in $\mathcal{Y}$. Suppose for a contradiction that there exists some graph $G$ in $\mathcal{F}_{HCU}$ but not in $\mathcal{Y}$. This implies the existence of a $\mathcal{Y}$-breaking pair. Let $d$ be the degree sequence of $G$, and since $G$ is not in $HCU$, there exists a Rao-minimal forbidden sequence $r$ that is Rao-contained by $d$. Thus there exists some realization $H$ of $r$ induced in a realization $G'$ of $d$. The class $\mathcal{Y}$ already includes every realization of the Rao-minimal sequences of $HCU$, so $H \in \mathcal{Y}$. By assumption, $G$ is $\mathcal{Y}$-free, but since $H$ is induced in $G'$, it follows that $G'$ is not $\mathcal{Y}$-free. Therefore $(G, G')$ is a $\mathcal{Y}$-breaking pair, so we know that $\mathcal{Y}$ is not DSF. 

Hence if $\mathcal{Y}$ is DSF, then $\mathcal{Y} = \mathcal{F}_{HCU}$, so it remains to show that $\mathcal{Y}$ is DSF. By Lemma \ref{prop:DSF_Barrus_Hartke}, it suffices to check degree sequences on 14 or fewer vertices. This is unwieldy, however - there are 29054155657235488 graphs on 14 vertices \cite{OEISA000088} - so we simplify the computational work in several ways. 

First, we argue that for degree sequences of non-split graphs, it suffices to check graphs up to 9 vertices. Let $\mathcal{Y_N}$ be the set of non-split graphs in $\mathcal{Y}$ and let $\mathcal{Y_S}$ be the set of split graphs in $\mathcal{Y}$. Since the graphs in Figure \ref{fig:big_graphs} are split graphs, every graph in $\mathcal{Y_N}$ realizes a Rao-minimal sequence on 7 or fewer terms. Let $(G,G')$ be a $\mathcal{Y}$-breaking pair of non-split graphs with decompositions $G = G_n \circ \ldots \circ G_0$ and $G' = G'_n \circ \ldots \circ G'_0$, and recall from Theorem \ref{prop:degree_sequence_decomposition} that $G_i$ and $G'_i$ have the same degree sequence. Since all graphs in $\mathcal{Y}$ are indecomposable, there exists $i$, $0 \le i \le n$, such that $(G_i,G'_i)$ is a $\mathcal{Y}$-breaking pair. We thus assume without loss of generality that $(G,G')$ is an indecomposable pair of non-split graphs where $G$ contains an induced element of $\mathcal{Y}$ and $G'$ does not. Then the degree sequence $\pi$ of $G$ and $G'$ is indecomposable, non-split, and not in $HCU$, so Corollary \ref{prop:nonsplit_contains_nonsplit} implies there exists a realization $G''$ of $\pi$ that contains an induced realization of a non-split Rao-minimal sequence. Therefore $(G',G'')$ is a $\mathcal{Y_N}$-breaking pair. By the proof of Lemma \ref{prop:DSF_Barrus_Hartke}, there exists a $\mathcal{Y_N}$-breaking pair on at most 9 vertices, which arises as we navigate via 2-switches from $G'$ to $G''$.
 
We use the algorithm shown at the end of the section to verify that each non-split degree sequence on $9$ or fewer vertices yields no $\mathcal{Y_N}$-breaking pairs. If a $\mathcal{Y_N}$-breaking pair were to exist, then it would contain as an induced subgraph a element of $\mathcal{F}_{HCU}$ but not $\mathcal{Y_N}$. Hence it is equivalent to check that the all non-split graphs in $\mathcal{F}_{HCU}$ on at most $9$ vertices are already elements of $\mathcal{Y_N}$. We know that all indecomposable graphs in $HCU$ on $5$ vertices are Rao-minimal, so beginning with $n=6$, we compute all forbidden degree sequences and check if their realizations include any graphs in $\mathcal{F}_{HCU}$. Every forbidden degree sequence on $n$ vertices is either Rao-minimal or greater than some forbidden degree sequence on $n-1$ vertices in the Rao ordering. We efficiently compute the set of forbidden degree sequences on $n$ vertices by taking the union of the non-split Rao-minimal sequences on $n$ vertices with the set of sequences on $n$ vertices that Rao-contain a forbidden sequence on $n-1$ vertices. From this list, we retain only the non-split sequences. By Corollary \ref{prop:MFIS_indecomp_closed}, all elements of $\mathcal{F}_{HCU}$ are indecomposable, so we further test only the indecomposable sequences. Corollary \ref{prop:MFIS_indecomp_closed} also implies that $\mathcal{F}_{HCU}$ is closed under complementation, so we keep only the first of these pairs of sequences. If a pair were $\mathcal{Y_N}$-breaking, its complement would be as well, so this allows us to test close to half as many sequences as we would otherwise. 

If a degree sequence $\pi$ on $n$ terms is the degree sequence of some $G \in \mathcal{F}_{HCU}$, then all induced subgraphs of $G$ on $n-1$ vertices are contained in $HCU$. Let $v(\pi)$ be the number of different terms in $\pi$; thus $G$ must have at least $v(\pi)$ non-isomorphic induced subgraphs on $n-1$ vertices in $HCU$. Before computing the realizations of $\pi$, we check for this by computing the number of graphic decrementations of $\pi$ in $HCU$.  This is counted by comparison with the list of forbidden degree sequences on $n-1$ vertices, since the criteria in Theorems \ref{prop:degree_characterization_nonsplit} and \ref{prop:degree_characterization_split} apply only to indecomposable sequences. If there are fewer than $v(\pi)$ such sequences, then no realization of $\pi$ is in $\mathcal{F}_{HCU}$. 

If there are at least $v(\pi)$ such sequences, then we compute all realizations of $\pi$, using an algorithm developed by Zolt\'{a}n Kir\'{a}ly \cite{Ki12}. For each realization $G$, we compute the degree sequence of each induced subgraph on $n-1$ vertices. If none of these are forbidden, then $G \in \mathcal{F}_{HCU}$. We add $G$ and $\overline{G}$ to the set of minimal forbidden induced subgraphs on $n$ vertices. 

Incrementing $n$, we repeat the entire above procedure until there are no minimal graphs for two consecutive values of $n$. The code terminates, returning $n$ and identifying all non-split graphs in $\mathcal{F}_{HCU}$. We ran this algorithm in Mathematica in twenty minutes. Fortunately, it returned $n=9$ and the minimal graphs were precisely the set $\mathcal{Y_N}$. 

We proceed to verify the same for $\mathcal{Y_S}$. While we do need to check degree sequences of order up to $14$, we know that any $\mathcal{Y}$-breaking pair of non-split graphs is a $\mathcal{Y_N}$-breaking pair, as shown above. Therefore, we need only check split degree sequences of orders $8$ to $14$. There are only 67997750 split graphs on $14$ vertices \cite{Ro00}, so this is a more manageable task. Our algorithm has three small differences from the non-split version: first, we calculate only split forbidden degree sequences. If $\pi$ is non-split, then all augmentations of $\pi$ are non-split, so in order to calculate the split forbidden degree sequences on $n$ vertices, we find the union of the set of augmentations of $\pi$ over all split forbidden degree sequences $\pi$ on $n-1$ vertices, then keep only the split sequences. Second, by Lemma \ref{prop:HCU_closed_complements}, a split graph $G$ is in $\mathcal{F}_{HCU}$ if and only if $\overline{G}, \GI,$ and $\overline{\GI}$ are as well. We group degree sequences with their inverses, complements, and inverse-complements, keeping one representative. Third and finally, all induced subgraphs of split graphs are split graphs, so when counting the number of graphic decrementations of $\pi$ in $HCU$, it suffices to compare to the list of forbidden split degree sequences. Since the algorithm is otherwise identical, it holds by the same proof. The program ran in two days and returned $14$, confirming that there are no $\mathcal{Y_S}$-breaking pairs of split graphs on $8$ to $14$ vertices. We therefore conclude that a graph $G$ is in $HCU$ if and only if it contains no element of $\mathcal{Y}$. 
\end{proof}

\noindent
{\ttfamily PROCEDURE \textsc{Minimal Forbidden Induced Subgraphs}}
\begin{adjustwidth}{0.5cm}{}
{\ttfamily FOR ALL} $n$, $5 \le n \le 8$, {\ttfamily SET} RaoMin($n$) to the set of Rao-minimal degree sequences of length $n$. \\
{\ttfamily IF} seeking non-split minimal forbidden induced subgraphs, {\ttfamily SET} $k=0$; \\
{\ttfamily IF} seeking split graphs, {\ttfamily SET} $k=1$.\\
{\ttfamily SET}  $n=5$ if $k =0$ and $n=8$ if $k = 1$. \\
{\ttfamily SET} Selected($n$) to the subset of non-split (if $k=0$) or split (if $k=1$) \\degree sequences in RaoMin($n$). \\
{\ttfamily SET}  Minimal($n$)  {\ttfamily TO} the set of graphs realizing sequences in RaoMin($n$). \\
{\ttfamily SET} Minimal($n-1$) {\ttfamily TO} $\emptyset$. \\
{\ttfamily UNTIL} Minimal($n-1$) and Minimal($n$) are each equal to $\emptyset$:  
\end{adjustwidth}
    \begin{adjustwidth}{1cm}{}
    {\ttfamily INCREMENT} $n$. \\    
    {\ttfamily SET}  Forbidden($n$) {\ttfamily TO}  the union of RaoMin($n$) and the set of augmentations of every degree sequence in Selected($n-1$), ordered lexicographically. \\
    {\ttfamily SET} Selected($n$) to the subset of non-split (if $k=0$) or split (if $k=1$) degree sequences \\ in Forbidden($n$).  \\        
    {\ttfamily SET}  Indecomposable($n$)  {\ttfamily TO} the subset of indecomposable sequences in Selected($n$). \\
    {\ttfamily SET}  Equivalent($n$)  {\ttfamily TO}  the subset of Indecomposable($n$) containing only the first member of each equivalence class under complementation (if $k=0$) or under complementation and inversion (if $k=1$). \\
    {\ttfamily FOR} each degree sequence $\pi$ in Equivalent($n$):
    \end{adjustwidth}
        \begin{adjustwidth}{1.5cm}{}
        {\ttfamily SET}  Values {\ttfamily TO}  the list of terms appearing in $\pi$, without repetition. \\
        {\ttfamily SET}  HCU-Down  {\ttfamily TO}  the set of graphic decrementations of $\pi$ not contained in Selected($n-1$). \\
        {\ttfamily IF} $|$HCU-Down$|$ $\ge$ $|$Values$|$:
        \end{adjustwidth}
            \begin{adjustwidth}{2cm}{}
            {\ttfamily SET}  Graphs {\ttfamily TO} the set of realizations of $\pi$. \\
            {\ttfamily FOR} each $G$ in Graphs:
            \end{adjustwidth}
                \begin{adjustwidth}{2.5cm}{}
                {\ttfamily IF} $G$ contains no induced subgraph on $n-1$ vertices with degree sequence in Selected(n-1):
                \end{adjustwidth}
                    \begin{adjustwidth}{3cm}{}
                    {\ttfamily APPEND} $G$ and $\overline{G}$, up to isomorphism, as terms in Minimal($n$). 
                    \end{adjustwidth}
                \begin{adjustwidth}{2.5cm}{}
                {\ttfamily ENDIF}
                \end{adjustwidth}
            \begin{adjustwidth}{2cm}{}
            {\ttfamily ENDFOR}
            \end{adjustwidth}
        \begin{adjustwidth}{1.5cm}{}
        {\ttfamily ENDIF}
        \end{adjustwidth}
    \begin{adjustwidth}{1cm}{}
    {\ttfamily ENDFOR}
    \end{adjustwidth}
\begin{adjustwidth}{0.5cm}{}
{\ttfamily RETURN} $n$ and the union of Minimal($5$) (if $k = 0$) or Minimal($8$) (if $k=1$) through Minimal($n$).
\end{adjustwidth}


\begin{thebibliography}{99}

%Barrus largest hereditary class within unigraphs
\bibitem{Ba12} M. D. Barrus,
\newblock On 2-switches and isomorphism classes,
\newblock \emph{Discrete Math.} {\bf 312} (2012), no. 15, 2217--2222.

%Barrus Havel-Hakimi residues paper
\bibitem{Ba12b} M.D. Barrus,
\newblock Havel-Hakimi residues of unigraphs, 
\newblock \emph{Information Processing Letters} {\bf 112} (2012), Iss. 1-2 January, 44--48.

%Barrus structural/deg seq characterization of largest hereditary class within unigraphs
\bibitem{Ba13} M. D. Barrus,
\newblock Hereditary unigraphs and the Erd\H{o}s-Gallai equalities,
\newblock \emph{Discrete Math.} {\bf 313} (2013), no. 21, 2469--2481.

%Barrus and Hartke paper (3) on DSF sets
\bibitem{BaHa15} M. D. Barrus and S. G. Hartke,
\newblock Minimal forbidden sets for degree sequence characterizations,
\newblock \emph{Discrete Math.} {\bf 338} (2015), no. 39, 1543--1554.

%A4 indecomposibility criterion
\bibitem{BaWe12} M. D. Barrus and D. B. West,
\newblock The $A_4$-structure of a graph,
\newblock \emph{J. Graph Theory} {\bf 71} (2012), no. 2, 159–175.

% Chudnovsky-Seymour on Rao
\bibitem{ChudnovskySeymour14} M. Chudnovsky and P. Seymour,
\newblock Rao's conjecture on degree sequences,
\newblock \emph{J. Combin. Theory Ser. B} {\bf 105} (2014), 44--92.

%Threshold graphs structure
\bibitem{ChHa75} V. Chv\'atal and P.L. Hammer,
\newblock Aggregation of inequalities in integer programming,
\newblock in Annals of Discrete Mathematics 1: Studies in Integer Programming, ed. P. L. Hammer, E. I. Johnson, B. H. Korte, and G. L. Nemhauser, North Holland (1977), 145-162.

%Balanced split graphs
\bibitem{CoTr18} K. L. Collins and A. N. Trenk,
\newblock Finding balance: split graphs and related classes,
\newblock \emph{Electron. J. of Combin.} {\bf 25} (2018), Article P1.73. 

%Split graph summary book chapter
\bibitem{CoTr21} K. L. Collins and A.N. Trenk,
\newblock Split graphs,
\newblock in \emph{Topics in algorithmic graph theory}, Encyclopedia of Mathematics and its Applications {\bf 178}, ed. L. Beineke, M. C. Golumbic, and R. Wilson, Cambridge University Press (2021), 189-206.

% Matrogenic graphs structure
\bibitem{FoHa78} S. F\"{o}ldes and P.L. Hammer,
\newblock On a class of matroid-producing graphs,
\newblock in Colloquia
Mathematica Societatis János Bolyai {\bf 18}, ed. A. Hajnal and V.T. Sós, North Holland (1978), 331-352.

%Split graphs definition and forbidden subgraph characterization
\bibitem{FoHa77} S. F\"{o}ldes and P. Hammer, 
\newblock Split graphs,
\newblock \emph{Congr. Numer.} {\bf 19} (1977), 311--315.

%Two-switch definition and theorem about co-realizations
\bibitem{FuHoMc65} D. R. Fulkerson, A. J. Hoffman, and M. H. McAndrew, 
\newblock Some properties of graphs with multiple edges, 
\newblock \emph{Canad. J. Math.} {\bf 17} (1965), 166--177. 

%Gale-Ryser inequality by Gale
\bibitem{Ga57} D. Gale, 
\newblock A theorem on flows in networks, 
\newblock \emph{Pac. J. Math.} {\bf 7} (1957), no. 2, 1073--1082.

%Gyarfas-Lehel definition of split graphs 
\bibitem{GyLe69} A. Gy\'{a}rf\'{a}s and J. Lehel, 
\newblock A Helly-type problem in trees, 
\newblock \emph{Combinatorial Theory and its Applications} {\bf 2} (1969), 571--584. 

%Deg seqs of threshold graphs - TG are unigraphs
\bibitem{HaIbSi81} P.L. Hammer, T. Ibaraki, and B. Simeone,
\newblock Threshold sequences,
\newblock \emph{SIAM J. Algebraic Discrete Methods} {\bf 2} (1981), 39--49.

%Degree sequence characterization of split graphs
\bibitem{HaSi81} P. Hammer and B. Simeone,
\newblock The splittance of a graph,
\newblock \emph{Combinatorica} {\bf 1} (1981), 275--284.

%Johnson unigraphs paper
\bibitem{Jo75} R.H. Johnson,
\newblock Simple separable graphs,
\newblock \emph{Pacific J. Math.} {\bf 56} (1975), 143--158.

%Johnson unigraphs survey
\bibitem{Jo80} R.H. Johnson,
\newblock Properties of unique realizations -- a survey,
\newblock \emph{Discrete Math.} {\bf 31} (1980), 185--192.

%Degree sequence realization algorithm
\bibitem{Ki12} Z. Kir\'{a}ly,
\newblock Recognizing graphic degree sequences and generating all realizations, 
\newblock Egervary Research Group Technical Report {\bf 11} (2011). 

%Kleitman Li background on unigraphic sequences
\bibitem{KlLi75} D.J. Kleitman and S.-Y. Li,
\newblock A note on unigraphic sequences,
\newblock \emph{Stud. Appl. Math.} {\bf LIV} (1975), 283--287.

%Kleitman and Wang realization algorithm
\bibitem{KlWa73} D.J. Kleitman and D.L.Wang, 
\newblock Algorithms for constructing graphs and digraphs with given valences and factors,
\newblock \emph{Discrete Math.} {\bf 6} (1973), 79--88.

%Koren background on unigraph sequences, 1
\bibitem{Ko76a} M. Koren,
\newblock Pairs of sequences with a unique realization by bipartite graphs,
\newblock \emph{J. Combin. Theory B} {\bf 21} (1976), 224--234.

%Koren background on unigraph sequences, 2
\bibitem{Ko76b} M. Koren,
\newblock Sequences with unique realization,
\newblock \emph{J. Combin. Theory B} {\bf 21} (1976), 235--244.

%Li background on unigraphic sequences
\bibitem{Li75} S.-Y. Li,
\newblock Graphic sequences with unique realization,
\newblock \emph{J. Combin. Theory B} {\bf 19} (1975), 42--68.

%Book on threshold sequences
\bibitem{MaPe95} N.V.R. Mahadev and U.N. Peled,
\newblock \emph{Threshold Graphs and Related Topics,}
\newblock in: Ann. Discrete Math. {\bf 56}, North-Holland, Amsterdam (1995).

%Marchioro et al matrogenics are unigraphs
\bibitem{MaEtAl84} P. Marchioro, A. Morgana, R. Petreschi, and B. Simeone,
\newblock Degree sequences of matrogenic graphs,
\newblock \emph{Discrete Math.} {\bf 51} (1984), 46--61.

%OEIS: Number of unlabeled graphs on n vertices
\bibitem{OEISA000088} Sequence A000088, The On-Line Encyclopedia of Integer Sequences, published electronically at oeis.org, (2010). 

%Peled matroidal
\bibitem{Pe77} U.N. Peled,
\newblock Matroidal graphs,
\newblock \emph{Discrete Math.} {\bf 20} (1977), 263--286.

%Rao order intro paper
\bibitem{Rao80} S.B. Rao,
\newblock Towards a theory of forcibly hereditary P-graphic sequences,
\newblock \emph{Combinatorics and Graph Theory: Proceedings of the Symposium Held at the Indian Statistical Institute, Calcutta, February 25--29.}
\newblock Springer-Verlag, \emph{Lecture Notes in Mathematics} {\bf 885} (1980), 441--458.

%Royle paper enumerating split graphs
\bibitem{Ro00} G. Royle, 
\newblock Counting set covers and split graphs, 
\newblock \emph{J. of Integer Sequences} {\bf 3} (2000), Article 00.2.6.

%Gale-Ryser inequality by Ryser
\bibitem{Ry57} H. J. Ryser, 
\newblock Combinatorial properties of matrices of zeros and ones,
\newblock \emph{Can. J. Math.} {\bf 9} (1957), 371--377. 

%Tyshkevich-Russian
\bibitem{Ty80} R. Tyshkevich, 
\newblock The canonical decomposition of a graph, 
\newblock \emph{Dokl. Akad. Nauk. BSSR} {\bf 24} (1980), 677--679 (in Russian).

%Tyshkevich
\bibitem{Ty00} R. Tyshkevich,
\newblock Decomposition of graphical sequences and unigraphs,
\newblock \emph{Discrete Math.} {\bf 220} (2000), 201--238.

%Tyshkevich on Matrogenic
\bibitem{Ty84} R.I. Tyshkevich,
\newblock Once more on matrogenic graphs,
\newblock \emph{Discrete Math.} {\bf 51} (1984), 91--100.

%Tyshkevich/Chernyak characterizations of unigraphs
\bibitem{TyCh78} R. Tyshkevich and A. Chernyak,
\newblock Unigraphs, I,
\newblock \emph{Vesti Akademii Navuk BSSR} {\bf 5} (1978), 5--11 (in Russian).

\bibitem{TyCh79a} R. Tyshkevich and A. Chernyak,
\newblock Unigraphs, II,
\newblock \emph{Vesti Akademii Navuk BSSR} {\bf 1} (1979), 5--12 (in Russian).

\bibitem{TyCh79b} R. Tyshkevich and A. Chernyak,
\newblock Unigraphs, III,
\newblock \emph{Vesti Akademii Navuk BSSR} {\bf 2} (1979), 5--11 (in Russian).

\bibitem{TyCh79c} R. Tyshkevich and A. Chernyak,
\newblock Canonical decomposition of a unigraph,
\newblock \emph{Vesti Akademii Navuk BSSR} {\bf 5} (1979), 14--26 (in Russian).


%Whitman thesis
\bibitem{Wh20} R. Whitman, 
\newblock Split graphs, unigraphs, and Tyshkevich decompositions, 
\newblock Undergraduate Honors Thesis, Wellesley College, Wellesley College Digital Repository (2020), repository.wellesley.edu/object/ir1225.

\end{thebibliography}
\end{document}